\documentclass{amsart}%
\usepackage{palatino, mathpazo}
\usepackage{amsfonts}
\usepackage{amsmath}
\usepackage{amssymb,latexsym,xcolor}
\usepackage{graphicx}
\usepackage{harpoon}
\usepackage[mathscr]{eucal}
\usepackage{amssymb}%
\usepackage[
linktocpage=true,colorlinks,citecolor=magenta,linkcolor=blue,urlcolor=magenta]{hyperref}
\providecommand{\U}[1]{\protect \rule{.1in}{.1in}}

\newtheorem{theorem}{Theorem}[section]

\newtheorem{corollary}[theorem]{Corollary}
\newtheorem{definition}[theorem]{Definition}

\newtheorem{lemma}[theorem]{Lemma}
\newtheorem{proposition}[theorem]{Proposition}
\newtheorem{Theorem}{Theorem}

\theoremstyle{remark}
\newtheorem{Remark}[theorem]{Remark}

\numberwithin{equation}{section}

\newcommand{\C}{\mathbb C}
\newcommand{\R}{\mathbb R}

\newcommand{\HH}{\mathbb H}
\newcommand{\dd}{\,\mathrm d}
\newcommand{\re}{\operatorname{Re}}
\newcommand{\im}{\operatorname{Im}}
\newcommand{\FP}{\operatorname{FP}}
\newcommand{\Res}{\operatorname{Res}}

\newcommand{\WW}{\mathcal W(B)}

\newcommand{\kk}{\mathcal K}
\newcommand{\Span}{\operatorname{span}_{\R}}

\newcommand{\p}{\partial}

\begin{document}
\title[Structure of critical points of multiple Green functions]{Structure of critical points of multiple Green functions on flat torus and applications}
\author{Zhijie Chen}
\address{Department of Mathematical Sciences, Yau Mathematical Sciences Center,
Tsinghua University, Beijing, 100084, China }
\email{zjchen2016@tsinghua.edu.cn}
\author{Chang-Shou Lin}
\address{Department of Mathematics, National Taiwan University, Taipei 10617, Taiwan }
\email{cslin@math.ntu.edu.tw}

\begin{abstract}
Let $E_{\tau}:=\mathbb{C}/(\mathbb{Z}+\mathbb{Z}\tau)$ be a flat torus and
$G(z)=G(z;\tau)$ be the Green function on $E_{\tau}$ with the singularity at $0$. Consider the multiple Green
function $G_{n}$ on $(E_{\tau})^{n}$:
\[
G_{n}(z_{1},\cdots,z_{n}):=\sum_{i<j}G(z_{i}-z_{j})-n\sum_{i=1}%
^{n}G(z_{i}).
\] A critical point is called trivial if $\{z_{1},\cdots,z_{n}\}=\{-z_{1},\cdots,-z_{n}\}$ in $E_{\tau}$. Lin and Wang (Ann. Math. 2010) proved that $G_1=-G$ has exactly three trivial critical points and at most one pair of nontrivial critical points (depends on the choice of $\tau$). For general $n\geq 2$, Chai, Lin ang Wang (Camb. J. Math. 2015) proved that $G_n$ has at most $2n+1$ trivial critical points, but how many nontrivial critical points might exist remained completely open there. 

In this paper, we prove that $G_n$ has at most $n$ pairs of nontrivial critical points, and this upper bound is optimal at least for $n=2,3$. The key idea is to show that nontrivial critical points (if exist) are always non-degenerate and contribute the same degree $(-1)^{n+1}$. Applications to the curvature equation
$\Delta u+e^{u}=8\pi n \delta_{0}$ and the integral Lam\'{e} equation $y''=[n(n+1)\wp(z)+B]y$ will be given. In particular, we show that the kernel of the linearized operator is one-dimensional for any solution of the curvature equation.

\end{abstract}

\maketitle

\section{Introduction}

\subsection{Main results}
Let $\tau \in \mathbb{H}:=\{ \tau \in \mathbb C\,|\operatorname{Im}\tau>0\}$, $\Lambda_{\tau}:=\mathbb{Z}+\mathbb{Z}\tau$ and $E_{\tau}:=\mathbb{C}/\Lambda_{\tau}$ be a flat torus in the plane. Denote also $$\omega_{1}=1
,\quad\omega_{2}=\tau,\quad\omega_{3}=1+\tau.$$ Let $G(z)=G(z;\tau)$ be
the Green function on $E_{\tau}$ defined by%
\[
-\Delta G(z)=\delta_{0}-\frac{1}{\left \vert E_{\tau}\right \vert }\text{
\ on }E_{\tau},\quad
\int_{E_{\tau}}G(z)=0,
\]
where $\delta_{0}$ is the Dirac measure at $0$ and $\vert E_{\tau}\vert$ is
the area of the torus $E_{\tau}$. It is an even function with the only
singularity at $0$. Given a positive integer $n\geq 1$, 
we define
\begin{equation}\label{eq:mU}
\mathcal U_{\tau}^n:=\big\{(z_{1},\cdot \cdot \cdot,z_{n})\in(E_{\tau})^{n}%
\,\big|\, z_j\neq 0, z_{j}\neq z_{k},\,\forall j\not =k\big\}/S_n,
\end{equation}
where $S_n$ is the permutation group.
Let $\boldsymbol{z}=\{z_{1},\cdot \cdot \cdot,z_{n}\}\in \mathcal U_{\tau}^n$.
In this paper, we continue to study the \emph{multiple Green function}
\begin{equation}
G_{n}(\boldsymbol{z})=G_{n}(\boldsymbol{z};\tau):=\sum_{j<k}G(z_{j}-z_{k};\tau)-n\sum_{j=1}%
^{n}G(z_{j};\tau),\label{multiple-G}%
\end{equation}
that was initiated by Chai, Lin and Wang \cite{CLW} and subsequently in \cite{CFL-Crelle, CL-JDG21, Lin-JDG25, LW2, LW-CAG}.
Note from $G(z)=G(-z)$ that $G_n$ is invariant under the permutation group $S_n$, hence is well-defined for $\boldsymbol{z}\in \mathcal U_{\tau}^n$. A $\boldsymbol{z}=\{z_1,\cdots,z_n\}$ is a critical point of $G_n$ if
\begin{equation}\label{poho}\nabla_{z_j}G_n(\boldsymbol{z})=\sum_{k=1,\not =j}^{n}\nabla G(z_{j}-z_{k})-n\nabla G(z_{j})=0,\qquad1\leq
j\leq n.\end{equation}
Clearly $-\boldsymbol{z}:=\{-z_1,\cdots,-z_n\}$ is also a critical point of $G_n$ if $\boldsymbol{z}$ is.

\medskip

\noindent{\bf Definition.} \cite{CLW}
{\it  A critical point $\boldsymbol{z}$ is called trivial if
\[\{z_1,\cdots,z_n\}=\{-z_1,\cdots,-z_n\}\quad\text{in }E_{\tau}.\]
Otherwise, the critical point $\boldsymbol{z}$ is called nontrivial.
Thus, nontrivial critical points must appear in pairs if exist.}

\medskip

Note that $G_1=-G$.
One of our motivations come from Lin and Wang \cite{LW}, where they studied the  number of critical points for the Green function $G(z)$. Since $G(z)$ is even and doubly periodic, it always
 has three trivial critical points $\frac{\omega_{k}}{2}$, $k=1,2,3$.
A basic question is how many nontrivial critical points $G(z)$ might have. This was solved in the following remarkable result.

\begin{Theorem} \cite{LW} \label{thm-LW} $G(z;\tau)$ has at most one pair of nontrivial critical points (depends on the choice of $\tau$).
 For example, 
\begin{itemize}
\item[(1)] when $\tau \in i\mathbb{R}_{>0}$, i.e., $\tau$ is purely imaginary and so $E_{\tau}$ is a rectangular torus,  $G(z;\tau)$ has only $3$ critical points  $\frac{\omega_k}{2}$, $k\in \{1,2,3\}$, which are all non-degenerate;
\item[(2)] while for $\tau=\frac{1}{2}+\frac{\sqrt{3}}{2}i$, i.e., $E_{\tau}$ is a rhombus torus, $G(z;\tau)$ has a unique pair of nontrivial critical points $\pm\frac{\omega_3}{3}$. 
\end{itemize}
 \end{Theorem}
Later, a complete characterization of those $\tau$'s such that $G(z;\tau)$
has a pair of nontrivial critical points was given in \cite{BE,CKLW}. Besides, it was proved in \cite{CFL,LW4} that the pair of nontrivial critical points of $G(z)$, if exist, must be non-degenerate minimal points.

For general $n\geq2$, $G_n$ was first studied by Chai, Lin and Wang \cite{CLW} by the method of algebraic geometry. Among other
things, they established the one-to-one correspondence between
trivial critical points of $G_{n}$ and branch points of a hyperelliptic curve
$Y_{n}$ concerning the integral Lam\'{e} 
equation
\begin{equation}\label{eq:Lame}
y^{\prime \prime}(z)=[n(n+1)\wp
(z)+B]y(z),\quad z\in\mathbb{C},
\end{equation}
where $n\in \mathbb{N}^*$ and $B\in \mathbb{C}$ are called index
and accessory parameter respectively, and $\wp(z)=\wp( z;\tau)$ is the Weierstrass elliptic function with periods
$\Lambda_{\tau}$, defined by%
\[
\wp(z;\tau):=\frac{1}{z^{2}}+\sum_{\omega \in \Lambda_{\tau
}\backslash\{0\}  }\left(  \frac{1}{(z-\omega)^{2}}-\frac
{1}{\omega^{2}}\right) ,
\]
which satisfies the well-known cubic equation
\[
\wp^{\prime}(z)^{2}=4\wp(z)^{3}-g_{2}\wp
(z)-g_{3}.
\]
Here $g_2=g_2(\tau), g_3=g_3(\tau)$ are known as invariants of the elliptic curve $E_{\tau}$. We will briefly review this result in Section \ref{sec 2}.
The Lam\'{e} equation \eqref{eq:Lame} has been widely studied in the literature;
see e.g. the classic text \cite{Whittaker-Watson} and recent works \cite{CLW,CKL-Dahmen,Dahmen0,Eremenko-CCM,Eremenko-GT,LW2} for introductions and recent developments.
In particular, Chai, Lin and Wang \cite{CLW} can count the number of trivial critical points of $G_n$ via the help of the Lam\'{e} equation \eqref{eq:Lame}.

\begin{Theorem}\cite{CLW}\label{Thm-B} $G_{n}(\boldsymbol{z};\tau)$ always has trivial critical points, and the number of trivial critical points is
 at most $2n+1$ (resp. exactly $2n+1$) for any
$\tau$ (resp. for almost all $\tau$ including any $\tau \in i\mathbb{R}_{>0}$).
\end{Theorem}

After Theorem \ref{Thm-B}, the remaining natural question is how many nontrivial critical points $G_n$ might have.
Recently, the authors \cite{CL-JDG21} gave a complete answer for $\tau \in i\mathbb{R}_{>0}$, i.e., $E_{\tau}$ is a rectangular torus.

\begin{Theorem}\cite{CL-JDG21}
\label{Thm-C} Let $\tau \in i\mathbb{R}_{>0}$. Then $G_{n}%
(\boldsymbol{z};\tau)$ has exactly
$2n+1$ critical points, which are all trivial critical points. Among them, there are exactly $n$ (resp. $n+1$) trivial critical points $\boldsymbol{a}$'s satisfying \[(-1)^n \det D^{2}G_{n}(\boldsymbol{a};\tau)<0\quad (\text{resp.}\; (-1)^n \det D^{2}G_{n}(\boldsymbol{a};\tau)>0).\] 
\end{Theorem}
Note that by writing $\boldsymbol{a}=\{a_1,\cdots,a_n\}$ with $a_k=x_{2k-1}+i x_{2k}$, then the Hessian
 \begin{align}\label{detDG}
 \det D^2G_n(\boldsymbol{a})=\det\left(\frac{\partial^2 G_n}{\partial{x_j}\partial x_k}\right)_{1\leq j,k\leq 2n}
 \end{align}
 is a determinant of the $2n\times 2n$ real matrix $D^2G_n(\boldsymbol{a})$, so $\det D^2G_n(\boldsymbol{a})$ is difficult to compute in general.

Theorem \ref{Thm-C} shows that $G_n$ has no nontrivial critical points when $\tau\in i\mathbb R_{>0}$, and can be seen as a generalization of Theorem \ref{thm-LW}-(1) to all $n\geq 2$. However, we can not expect the non-existence of nontrivial critical points for general $\tau\in\mathbb H$. Indeed, 
let us define
\begin{align}\label{eq-en}\mathcal{E}_n:=&\{\tau\in\mathbb H\;:\; G_n(\boldsymbol{z};\tau)\;\text{has nontrivial critical points}\},
\end{align}
then among other things, Lin \cite{Lin-JDG25} proved that $\mathcal{E}_n\neq\emptyset$ is an unbounded and open subset in $\mathbb{H}$. Unfortunately,  it is still unknown from \cite{Lin-JDG25} that how many nontrivial critical points $G_n(\boldsymbol z;\tau)$ might have for $\tau\in\mathcal{E}_n$. The best result on this aspect is Eremenko's \cite[Theorem 1]{Eremenko}, which showed that the number of those $B$'s such that the monodromy of the Lam\'{e} equation \eqref{eq:Lame} is unitary is always finite. By Theorem \ref{Thm-CLW}-(2) below, this statement is equivalent to say that the number of nontrivial critical points of $G_n(\boldsymbol z;\tau)$ is always finite. Eremenko wrote in \cite[p.3959]{Eremenko} that ``Unfortunately, our proof is non-constructive, and does not give any explicit upper
estimate for the number''.

The main result of this paper to give the first explicit upper bound of the number of nontrivial critical points for all $\tau\in\mathbb H$, hence generalizing the main statement of Theorem \ref{thm-LW} to all $n\geq 2$.

\begin{theorem}\label{main-thm-1} Let $n\geq 2$ and $\tau\in\mathbb H$. Then $G_n(\boldsymbol z;\tau)$ has at most $n$ pairs of nontrivial critical points (depends on the choice of $\tau$). Furthermore, for any nontrivial critical point $\boldsymbol {a}$, there holds
\begin{equation}\label{eq:nontri-degree}
(-1)^{n+1} \det D^{2}G_{n}(\boldsymbol{a};\tau)>0.
\end{equation}
Thus, nontrivial critical points, if exist, must be non-degenerate and contribute the same degree $(-1)^{n+1}$ to the topological degree of $\nabla G_n=0$.
\end{theorem}

\begin{Remark} Here are some comments for Theorem \ref{main-thm-1}.
\begin{itemize}
\item[(1)] It is unexpected to us that \eqref{eq:nontri-degree} holds for all nontrivial critical points. In \cite{CFL}, Fu and the authors proved that nontrivial critical points of the sum $G(z+p)+G(z-p)$ could be degenerate for some choice of $\tau$ and $p\in E_{\tau}$. Therefore, it is quite surprising to us that nontrivial critical points of $G_n$ are  always non-degenerate (if exist).
\item[(2)] Obviously, one can not expect that any trivial critical point is always non-degenerate either (Otherwise, the number of critical points of $G_n(\boldsymbol z;\tau)$ is the same number for all $\tau\in\mathbb H$). Indeed, it was proved in \cite{Lin-JDG25} that the set
\[\big\{\tau\in \mathbb H\;:\; \text{$G_n(\boldsymbol{z};\tau)$ has degenerate trivial critical points}\big\}\]
contains countably many analytic curves that were called Lin-Wang curves by \cite{Eremenko-CCM}. This is the reason why the number of critical points of $G_n(\boldsymbol z;\tau)$ could change in a complicated way when $\tau\in\mathbb H$ varies.
\item[(3)] We can see from Theorem \ref{Thm-C} and Corollary \ref{coro1-5} that for each $k\in\{0,1,2\}$, there exists $\tau$ such that $G_2(\boldsymbol z;\tau)$ has exactly $k$ pairs of nontrivial critical points. Furthermore, we will prove in Section \ref{sec n=3} that $G_3(\boldsymbol z;\tau)$ has exactly $3$ pairs of nontrivial critical points for $\tau=\frac{1}{2}+\frac{\sqrt{3}}{2}i$.
Thus, the upper bound $n$ in Theorem \ref{main-thm-1} is optimal at least for $n=2,3$. For general $n\geq 4$, whether this upper bound $n$ is optimal or not seems very challenging and remains open. We might study this problem elsewhere.
\end{itemize}
\end{Remark}

Our proof of Theorem \ref{main-thm-1} consists of two key steps. The first step is to prove \eqref{eq:nontri-degree} for all nontrivial critical points. In \cite{CFL-Crelle}, Fu and the authors proved a precise formula for the Hessian of a nontrivial critical point $\boldsymbol a$:
\begin{equation}\label{eq-ex-0non}
\det D^2 G_n(\boldsymbol{a}) = \frac{(-1)^n n^2c_{\boldsymbol{a}}}{4(2\pi)^{2n+2} \operatorname{Im}\tau}  \operatorname{Im}
\big((\overline{\tau_r}\tau_s)(r_{\boldsymbol{a}}, s_{\boldsymbol{a}})\big), 
\end{equation}
where $c_{\boldsymbol{a}}>0$ and $\tau(r,s)$ is an analytic function that will be recalled in Sections \ref{sec 2}-\ref{sec 4}.  
See also Theorem \ref{DD-conj} for this formula. Together with Sard theorem, this formula was used in \cite{CFL-Crelle} to prove that for almost all $\tau$, all critical points of $G_n(\boldsymbol z;\tau)$ are non-degenerate. The novelty of this paper is that, with the help of finite-part integrals studied in Section \ref{sec:finite}, we can prove a surprising stronger result that 
\begin{equation}\label{eq:taurtaus}\operatorname{Im}
\big((\overline{\tau_r}\tau_s)(r_{\boldsymbol{a}}, s_{\boldsymbol{a}})\big)<0\end{equation}
always holds, which guarantees the validity of \eqref{eq:nontri-degree} for all nontrivial critical points.
The second key step is to prove that the total topological degree of $\nabla G_n=0$ is always $(-1)^n$; see Theorem \ref{thm-degree}. Since  \eqref{eq:nontri-degree} shows that every nontrivial critical point contributes the same degree $(-1)^{n+1}$, together with the possible degrees of trivial critical points, we can prove that the number of pairs of nontrivial critical points is at most $n$. Remark that if we can only prove the non-degeneracy of all nontrivial critical points rather than the stronger statement \eqref{eq:nontri-degree}, then the degrees of different nontrivial critical points
may cancel with each other, so we can not control the number of nontrivial critical points by this degree approach.

\subsection{Applications} 
Our another motivation of studying $G_{n}$ comes from its deep connection (see e.g.
\cite{BKLY,BT, CLW, CL-1, CL-JDG21, Lin-JDG25, LW-CAG, WeiZhang2026}) with the bubbling phenomenon of
 the curvature equation:
\begin{equation}
\Delta u+e^{u}=8\pi n \delta_{0}\quad\text{  on }\;E_{\tau}.\label{mfe}%
\end{equation}
Geometrically, a solution $u$ leads to a spherical metric $ds^{2}=\frac{1}{2}%
e^{u}|dz|^2$ with constant Gaussian curvature $+1$ acquiring a conic
singularity at $0$. Spherical metrics with conic singularities on compact Riemann surfaces have received great interest for a long time; see e.g. \cite{BMM, KZ, Kuo, MZ-IMRN, MP-IMRN, MP-GAFA, WWX} and the references therein for recent developments on this subject. It also appears in statistical physics from
the mean field limit of the Euler flow in Onsager's vortex model (cf.
\cite{CLMP}), hence also called a mean field equation; see e.g. \cite{CLW, LW, LW2} and the references therein.

The existence of solutions of \eqref{mfe} is very challenging from
the PDE point of view. In fact, the solvability of \eqref{mfe}
essentially depends on the moduli $\tau$ in a sophisticated manner. This
phenomenon was first discovered by Lin and Wang \cite{LW} when
they studied the case $n=1$.
Chai, Lin and Wang \cite{CLW} studied the general case $n\geq 2$
by the method of algebraic geometry, and established deep connections among the curvature equation \eqref{mfe}, the multiple Green function $G_n$ and the Lam\'{e} equation \eqref{eq:Lame}, as stated as follows. See also their sequel work \cite{LW2} for  other connections among these three subjects and modular forms.

\begin{Theorem}\cite{CLW}\label{Thm-CLW} Let $n\geq 1$ and $\tau\in\mathbb H$.
\begin{itemize}
\item[(1)] Let $u(z)$ be a solution of \eqref{mfe}. Then there is a meromorphic function $f(z)$ (called a developing map) satisfying $f(-z)=1/f(z)$ and
\begin{equation}
f(z+\omega_j)=e^{2\pi i\theta_j}f(z),\quad (\theta_1,\theta_2)\in\mathbb R^2\setminus\mathbb Z^2, \quad j=1,2,
\end{equation}
such that the following statements hold:
\begin{itemize}
\item[(1-1)] The $u_{\lambda}(z)$ defined by
\begin{equation}\label{eq:scaling}
 u_\lambda(z):=\log\frac{8e^{2\lambda}|f'(z)|^2}
 {(1+e^{2\lambda}|f(z)|^2)^2},\qquad \forall\lambda\in\mathbb R,
\end{equation}
is a one-parameter (scaling) family of solutions to \eqref{mfe}, and there is $\lambda_0\in\mathbb R$ such that $u(z)=u_{\lambda_0}(z)$. Furthermore, $u_{\lambda}(z)=u_{\lambda}(-z)$ if and only if $\lambda=0$.
\item[(1-2)] As $\lambda\to +\infty$, $u_{\lambda}(z)$ blows up exactly at $n$ points $\{a_1,\cdots,a_n\}\subset E_{\tau}\setminus\{0\}$, and this $\boldsymbol{a}=\{a_1,\cdots,a_n\}$ gives a nontrivial critical point of $G_n$. Furthermore, $u_{\lambda}(z)$ blows up exactly at $\{-a_1,\cdots,-a_n\}$ as $\lambda\to-\infty$.
\item[(1-3)] Define
$$B_{\boldsymbol{a}}:=(2n-1)\sum_{k=1}^n\wp(a_k).$$
Then the monodromy of the Lam\'{e} equation \eqref{eq:Lame} with $B=B_{\boldsymbol a}$ is unitary, i.e., up to a common conjugation, the monodromy group of \eqref{eq:Lame} with $B=B_{\boldsymbol a}$ is a subgroup of $\mathrm{SU}(2)$.
\end{itemize}
\item[(2)] There is a one-to-one correspondence between any two of the following:
\begin{itemize}
\item[(2-1)] the one-parameter families of solutions (or equivalently, even solutions) of the curvature equation \eqref{mfe};
\item[(2-2)] the pairs of nontrivial critical points of the multiple Green function $G_n$;
\item[(2-3)] those $B$'s such that the monodromy of the Lam\'{e} equation \eqref{eq:Lame} is unitary.
\end{itemize}
\end{itemize}
\end{Theorem}

Here we say a sequence of solutions $u_m(z)$ blow up at a point $a_0$, if there is $p_m\to a_0$ such that $u_m(p_m)\to +\infty$ as $m\to+\infty$.
In view of Theorem \ref{Thm-CLW}, Theorem \ref{main-thm-1} has interesting applications to the curvature equation and the Lam\'{e} equation.

\begin{corollary}\label{coro1-4}
Let $n\geq 2$ and $\tau\in\mathbb H$. Then the following statements hold.
\begin{itemize}
\item[(1)] The curvature equation \eqref{mfe} has at most $n$ one-parameter families of solutions (or equivalently, at most $n$ even solutions).
\item[(2)] There are at most $n$ $B$'s such that the monodromy of the Lam\'{e} equation \eqref{eq:Lame} is unitary.
\end{itemize}
\end{corollary}

\begin{proof}
This follows directly from Theorem \ref{main-thm-1} and Theorem \ref{Thm-CLW}.
\end{proof}

Corollary \ref{coro1-4}-(2) improves Eremenko's finiteness result \cite[Theorem 1]{Eremenko} for the Lam\'{e} equation \eqref{eq:Lame} to the first explicit upper bound.

Clearly Theorem \ref{Thm-C} and Theorem \ref{Thm-CLW}-(2) imply that the curvature equation \eqref{mfe} has no solutions if $\tau\in i\mathbb R_{>0}$; while \eqref{eq-en} and Theorem \ref{Thm-CLW}-(2) imply that \eqref{mfe} has even solutions if $\tau\in\mathcal E_n$. Here we can obtain some precise existence results for the case $n=2$.

\begin{corollary}\label{coro1-5}
Let $\tau=\frac12+ib$ with $b>0$, i.e., $E_{\tau}$ is a rhombus torus. Let $n=2$, i.e.,
\begin{equation}
\Delta u+e^{u}=16\pi \delta_{0}\quad\text{ on }\;E_{\tau}.\label{mf2e}%
\end{equation}
\begin{itemize}
\item[(1)] There is $b^*>\sqrt{3}/2$ such that for any $b>b^*$, \eqref{mf2e} has a unique even solution. Consequently, $G_2(\boldsymbol z;\tau)$ has a unique pair of nontrivial critical points.

\item[(2)] There is $b\in (1/(2\sqrt{3}), \sqrt{3}/2)$ such that \eqref{mf2e} has exactly two even solutions. Consequently, $G_2(\boldsymbol z;\tau)$ has exactly two pairs of nontrivial critical points.
\end{itemize}
\end{corollary}

\begin{proof}
(1) Applying Eremenko-Gabrielov's general result \cite[Theorem 1.5]{EG} to \eqref{mf2e}, we obtain the existence of $b^*>\sqrt{3}/2$ such that for any $b>b^*$, \eqref{mf2e} has a unique even axisymmetric solution. Suppose \eqref{mf2e} has another even solution $u(z)$, then $u(z)$ can not be axisymmetric, i.e., $u(z)\neq u(\bar z)$ by using the complex variable $z=x+iy$. This implies that $\tilde u(z):=u(\bar z)$ is the third even solution of \eqref{mf2e}, a contradiction with Corollary \ref{coro1-4}. Therefore, the even solution of \eqref{mf2e} is unique.

(2) Fu and the authors \cite[Theorem 1.6]{CFL21} proved the existence of  $b\in (\frac{1}{2\sqrt{3}}, \frac{\sqrt{3}}{2})$ such that \eqref{mf2e} has no even axisymmetric solutions, but has at least two even solutions which are not axisymmetric. Together with Corollary \ref{coro1-4}, we conclude that \eqref{mf2e} has exactly these two even solutions.
\end{proof}

Finally, given any solution $u(z)$ of the curvature equation \eqref{mfe}, we study distributional solutions of the linearized equation
\begin{equation}\label{eq:linear}\Delta \varphi+e^u\varphi=0\quad\text{on }\; E_{\tau}.\end{equation}
Since 
$$u(z)=4 n\log|z|+O(1)\quad\text{near }z=0$$
implies that $e^{u(z)}$ extends smoothly across $z=0$, the elliptic regularity theory implies that any distributional solution of \eqref{eq:linear} satisfies $\varphi\in C^2(E_{\tau})$. Define
\begin{equation}\label{eq:kernel}
 \kk_u:=\{\varphi\in C^2(E_\tau):
                    (\Delta+e^u)\varphi=0\},
\end{equation}
\[ \kk_u^{\mathrm{ev}}
 :=\{\varphi\in\kk_u:\varphi(-z)=\varphi(z)\}.\]
Clearly $ \kk_u^{\mathrm{ev}}=\emptyset$ if $u(z)$ is not even.
By Theorem \ref{Thm-CLW}-(1-1), there is a one-parameter family of solutions $u_{\lambda}(z)$ of the form \eqref{eq:scaling} such that $u(z)=u_{\lambda_0}(z)$ for some $\lambda_0\in\R$. Then a direct computation shows that $\partial_\lambda u_\lambda\big|_{\lambda=\lambda_0}$ is a solution of the 
linearized equation \eqref{eq:linear}, so $\partial_\lambda u_\lambda\big|_{\lambda=\lambda_0}\in\kk_u$. Our final result shows that this gives all the solutions of \eqref{eq:linear} up to multiplying a constant.

\begin{theorem}\label{thm:Jacobi}
Let $n\ge1$ and $\tau\in\HH$. Given a one-parameter family of solutions $u_{\lambda}(z)$ of the form \eqref{eq:scaling} for the curvature equation \eqref{mfe}. Then 
\begin{equation}\label{eq:fullkernel}
 \kk_{u_\lambda}
 =\Span\{\partial_\lambda u_\lambda\},\qquad
\forall\lambda\in\R.
\end{equation}
Furthermore, recalling that $u_0(z)$ is even, since $\partial_\lambda u_\lambda\big|_{\lambda=0}$ is odd,
it follows that
\begin{equation}\label{eq:evenkernel}
 \kk_{u_0}^{\mathrm{ev}}=\{0\}.
\end{equation}
\end{theorem}
In \cite{LW}, Lin and Wang proved \eqref{eq:evenkernel} for even solutions of $\Delta u+e^u=\rho\delta_0$ with $\rho\in [4\pi, 8\pi]$, but did not prove  \eqref{eq:fullkernel} for $\Delta u+e^u=8\pi\delta_0$. Thus, the statement \eqref{eq:fullkernel} is new even for $n=1$.

The rest of this paper is organized as follows. In Section \ref{sec 2}, we briefly review the deep connection between critical points of $G_n$ and the monodromy of the Lam\'{e} equation from \cite{CLW}. In Section \ref{sec:finite}, we study the finite-part integrals for some functions related to the Lam\'{e} equation and prove an important strict inequality that is essentially equivalent to \eqref{eq:taurtaus}. In Section \ref{sec 4}, we apply this inequality to prove Theorem \ref{main-thm-1}. In Section \ref{sec:jacobi}, we apply this inequality to study the linearized equation and prove Theorem \ref{thm:Jacobi}. Finally, in Section \ref{sec n=3}, we consider the case $n=3$ and prove that $G_3(\boldsymbol z;\tau)$ has exactly $3$ pairs of nontrivial critical points for $\tau=\frac{1}{2}+\frac{\sqrt{3}}{2}i$. 

The approach of this paper can be developed further to study the curvature equation with four singular sources
$$\Delta u+e^u=8\pi n_0\delta_0+8\pi\sum_{k=1}^3n_k\delta_{\frac{\omega_k}{2}}\quad\text{on }E_{\tau},$$ 
which will be studied in a sequel work.

\section{Connection with the Lam\'{e} equation}
\label{sec 2}

In this section, we
briefly review the connection between critical points of $G_{n}$ and
the classical Lam\'{e} equation from \cite{CLW}. 

Let $\zeta(z)=\zeta(z;\tau):=-\int^{z}\wp(\xi;\tau)d\xi$ be the Weierstrass
zeta function with two quasi-periods
\begin{equation}
\eta_j=\eta_{k}(\tau):=2\zeta(\tfrac{\omega_{k}}{2};\tau)=\zeta(z+\omega_{j}%
;\tau)-\zeta(z;\tau),\quad k=1,2,\label{40-2}%
\end{equation}
and $\sigma(z)=\sigma(z;\tau):=\exp(\int^{z}\zeta(\xi;\tau)d\xi)$ be the Weierstrass sigma function. Notice that $\zeta(z)$ is an odd meromorphic function with simple poles at $\Lambda_{\tau}$, and $\sigma(z)$ is an odd
entire function with simple zeros at $\Lambda_{\tau}$, which satisfies the following transformation law
\begin{equation}\label{40-3}
\sigma(z+\omega_{k})=-e^{\eta_{k}(z+\omega_{k}/2)}\sigma(z),\quad k=1,2.
\end{equation}

\subsection{Lam\'{e} equation}

Consider the integral Lam\'{e} equation
\begin{equation}\label{eq:lame}
y^{\prime \prime}(z)=[n(n+1)\wp
(z;\tau)+B]y(z),\quad z\in\mathbb{C},
\end{equation}
where $n\in \mathbb{N}^*$ and $B\in \mathbb{C}$. Equation \eqref{eq:lame} has a regular singularity at $z=0$ with local exponents $-n$ and $n+1$. Thanks to that $\wp(z)$ is even,
it is easy to see (cf. \cite{CLW,Dahmen0,Whittaker-Watson}) that any solution of \eqref{eq:lame} has no logarithmic singularities, and hence is single-valued and meromorphic in $\mathbb{C}$ after analytic contiuation. 

\begin{lemma}\label{lem:frobenius}
Write
\begin{equation}\label{eq:wpseries}
 \wp(z)=z^{-2}+\sum_{j=1}^\infty v_jz^{2j},\quad\text{where }v_j\in\mathbb Q[g_2, g_3].
\end{equation}
Then \eqref{eq:lame} has a local basis
{\allowdisplaybreaks
\begin{align}\label{eq:frob}
 &\varphi(z)=\varphi(z;B)=z^{-n}\sum_{k=0}^\infty a_k(B)z^{2k}, \quad a_0=1,\\
 &\psi(z)=\psi(z; B)=z^{n+1}\sum_{k=0}^\infty b_k(B)z^{2k},
 \quad b_0=1,\nonumber
\end{align}
}%
where, for $k\ge1$, $a_k(B)$ and $b_k(B)$ are uniquely solved as polynomials of $B$ by the following recursive formulas
{\allowdisplaybreaks
\begin{align*}
 2k(2k-2n-1)a_k
 &=Ba_{k-1}+n(n+1)\sum_{j=1}^{k-1}v_j a_{k-j-1},\\
 2k(2k+2n+1)b_k
 &=Bb_{k-1}+n(n+1)\sum_{j=1}^{k-1}v_j b_{k-j-1}.
\end{align*}
}%
Every solution extends to a single-valued meromorphic function on
$\C$, with possible poles only at $\Lambda_\tau$.
\end{lemma}

\begin{proof}
The proof is standard by using Frobenius' method; see e.g. \cite{CLW,Dahmen0,Whittaker-Watson}.
\end{proof}

For $\boldsymbol{a}=\{a_{1},\cdot \cdot \cdot,a_{n}\}\subset E_{\tau}\setminus\{0\}$, we define
\begin{equation}\label{eq:Ac}
 \qquad c(\boldsymbol a):=\sum_{j=1}^n\zeta(a_j),
\end{equation}
and consider the normalized
Hermite-Halphen ansatz (cf. \cite{CLW,Dahmen0,Whittaker-Watson}):
\begin{equation}
y_{+}(z):=y_{\boldsymbol{a}}(z):=e^{c(\boldsymbol a)z}\prod_{j=1}%
^{n}\frac{\sigma(z-a_{j})}{\sigma(-a_j)\sigma(z)}=\frac{1}{z^n}+\cdots,\label{ff2-1}%
\end{equation}
then
\begin{align}
y_{-}(z):=&y_{-\boldsymbol{a}}(z):=e^{c(-\boldsymbol a)z}\prod_{j=1}%
^{n}\frac{\sigma(z+a_{j})}{\sigma(a_j)\sigma(z)}=\frac{1}{z^n}+\cdots\nonumber\\
=&(-1)^ny_{+}(-z).\label{eq:reflection}
\end{align}
Following also \cite{CLW}, we define
\begin{equation}
\label{hyper}Y_{n}=Y_n(\tau):=\left \{  \boldsymbol{a}\in E_{\tau}^n/S_n\left \vert
\begin{array}
[c]{l}%
a_{j}\neq0,\text{ }a_{j}\neq a_{k}\text{ in $E_{\tau}$ for any }j\not =k,\\
\sum_{k\not =j}^{n}(\zeta(a_{j}-a_{k})+\zeta(a_{k})-\zeta(a_{j}))=0,\text{
}\forall j
\end{array}
\right.  \right \}  .
\end{equation}
Clearly $-\boldsymbol{a}:=\{-a_{1},\cdot \cdot \cdot,-a_{n}\} \in Y_{n}$ if
$\boldsymbol{a}\in Y_{n}$. The connection between the Lam\'{e} equation and $Y_n$ 
has been studied in the literature.

\begin{theorem}
(cf. \cite{CLW,Whittaker-Watson}) \label{CLWWW}
\begin{itemize}
\item[(1)]
For any $B\in\mathbb C$, there exists a unique pair $\pm\boldsymbol{a}\in Y_n$ such that $y_{\pm}(z)=y_{\pm\boldsymbol{a}}(z)$ are solutions of the Lam\'{e} equation \eqref{eq:lame}, and
\begin{equation}
\label{b-a}B=B_{\boldsymbol{a}} :=(2n-1)\sum_{j=1}^{n}\wp(a_{j}).
\end{equation}
\item[(2)] There is a so-called Lam\'e polynomial $\ell_{n}(B)=\ell_{n}(B;\tau
)\in \mathbb{Q}[g_{2}(\tau),g_{3}(\tau)][B]$ of degree $2n+1$ such that $y_{+}(z)$ and $y_{-}(z)$ are linearly independent if and only if $\ell_n(B)\neq 0$. In this case, $\boldsymbol{a}\neq -\boldsymbol{a}$ and indeed,
\begin{equation}
\{a_{1},\cdot \cdot \cdot,a_{n}\} \cap \{-a_{1},\cdot \cdot \cdot,-a_{n}%
\}=\emptyset\quad\text{in }\;E_{\tau}.\label{fc6}%
\end{equation}
Besides, we define $(r,s)\in\mathbb{C}^2$ by
\begin{equation}\label{de-rs}\left \{
\begin{array}
[c]{l}%
r+s\tau=\sum_{j=1}^{n}a_j\\
r\eta_{1}+s\eta_{2}=c(\boldsymbol a)=\sum_{j=1}^{n}\zeta(a_{j}),
\end{array}
\right.\end{equation}
which can be uniquely solved by the Legendre relation $\tau\eta_1-\eta_2=2\pi i$ as
\begin{equation}\label{eq:rs-ex}
-2\pi i s=\sum_{j=1}^{n}\zeta(a_{j})-\eta_1\sum_{j=1}^{n}a_j,\quad 2\pi i r=\tau \sum_{j=1}^{n}\zeta(a_{j})-\eta_2\sum_{j=1}^{n}a_j.
\end{equation}
Then $(r,s)\notin\frac12\mathbb{Z}^2:=\{(r,s) \;:\; 2r,2s\in\mathbb Z\}$ and
\begin{equation}\label{mono1}\begin{pmatrix}y_{+}(z+\omega_1)\\ y_{-}(z+\omega_1)\end{pmatrix}=\begin{pmatrix}
e^{-2\pi is} & 0\\
0 & e^{2\pi is}%
\end{pmatrix}\begin{pmatrix}y_{+}(z)\\ y_{-}(z)\end{pmatrix},\end{equation}
\begin{equation}\label{mono2}\begin{pmatrix}y_{+}(z+\omega_2)\\ y_{-}(z+\omega_2)\end{pmatrix}=\begin{pmatrix}
e^{2\pi ir} & 0\\
0 & e^{-2\pi ir}%
\end{pmatrix}\begin{pmatrix}y_{+}(z)\\ y_{-}(z)\end{pmatrix}.\end{equation}
\item[(3)] The
map $B:Y_{n}\rightarrow \mathbb{C}$ defined by \eqref{b-a} is a ramified
covering of degree $2$, and 
\[
Y_{n}\cong \left \{  (B,C)\text{ }|\text{ }C^{2}=\ell_{n}(B)\right \}
\]
 is a
hyperelliptic curve with $\boldsymbol{a}\in Y_{n}$ being a branch
point if $\boldsymbol{a}=-\boldsymbol{a}$ in $E_{\tau}$.
\item[(4)] The Lam\'{e} polynomial $\ell_n(\cdot;\tau)$ has $2n+1$ distinct roots for almost all $\tau$'s (including any $\tau\in i\mathbb{R}_{>0}$).
\end{itemize}
\end{theorem}



Remark that \eqref{mono1}-\eqref{mono2} follow directly from \eqref{40-3}, \eqref{ff2-1}-\eqref{eq:reflection} and \eqref{eq:rs-ex}, and $-\boldsymbol{a}$ corresponds to $(-r, -s)$. Note also that if $\tilde{a}_j\in a_j+\Lambda_{\tau}$, then $$\tilde{\boldsymbol{a}}:=\{\tilde{a}_1,\cdots,\tilde{a}_n\}=\boldsymbol{a}\quad\text{in }E_{\tau},$$ and $y_{\pm\tilde{\boldsymbol{a}}}(z)=C_{\pm}y_{\pm\boldsymbol a}(z)$ for some constants $C_{\pm}\neq 0$, and the corresponding $(\tilde{r},\tilde{s})$ determined by $\tilde{\boldsymbol a}$ via \eqref{de-rs} satisfies $(\tilde{r},\tilde{s})\equiv (r,s)$ mod $\mathbb{Z}^2$. Thus, $(r,s)\in(\mathbb{C}^{2}\backslash \frac{1}{2}\mathbb{Z}^{2})/\sim$ is uniquely determined by $(\tau,B)$ and is known as the monodromy data of \eqref{eq:lame} (cf. \cite{CLW, CL-AIM25, LW2}), because the monodromy group of \eqref{eq:lame} is precisely generated by the two matrices in \eqref{mono1}-\eqref{mono2}. Here $\sim$ is an equivalent relation defined by\[(r,s)\sim (\tilde r,\tilde s)\quad\text{ if}\quad (r,s)\equiv \pm (\tilde r,\tilde s)\quad\operatorname{mod}\; \mathbb{Z}^2. \]

Define
\begin{equation}
\Sigma_{n}:=\big\{(\tau, B)\in \mathbb{H}\times \mathbb{C}\,|\, \ell_n(B;\tau)\neq 0\big\},
\end{equation}
which is clearly an open connected subset of $\mathbb{H}\times \mathbb{C}$.
Then by the above argument, the map $\Sigma_n\ni (\tau, B)\mapsto (r,s)\in (\mathbb{C}^{2}\backslash \frac{1}{2}\mathbb{Z}^{2})/\sim$ is well-defined. 

Given any $(\tau_0, B_0)\in \Sigma_{n}$, take $(r_0,s_0)\in \mathbb{C}^{2}\backslash \frac{1}{2}\mathbb{Z}^{2}$ to be a representative of  the monodromy data of \eqref{eq:lame} with $(\tau,B)=(\tau_0,B_0)$. Since there is a small neighborhood $V\subset \mathbb{C}^{2}\backslash \frac{1}{2}\mathbb{Z}^{2}$ of $(r_0,s_0)$ such that $(r,s)\not\sim (r',s')$ for any $(r,s), (r',s')\in V$ satisfying $(r,s)\neq (r',s')$, there is a small neighborhood $U\subset \Sigma_{n}$ of $(\tau_0, B_0)$ such that $U\ni(\tau,B)\mapsto (r,s)\in(\mathbb{C}^2\setminus\tfrac12\mathbb{Z}^2)/\sim$ can be seen as 
\begin{equation}\label{eq-UV} U\ni(\tau,B)\mapsto (r,s)\in V\subset \mathbb{C}^2\setminus\tfrac12\mathbb{Z}^2,\end{equation}
and so we can consider the local analytic property of this map. Recently, the authors \cite{CL-AIM25} proved a surprising universal law for this map.

\begin{theorem}\cite{CL-AIM25}\label{thm-rsbtau} The map  $\Sigma_n\ni (\tau, B)\mapsto (r,s)\in (\mathbb{C}^{2}\backslash \frac{1}{2}\mathbb{Z}^{2})/\sim$ is holomorphic and locally one-to-one, and satisfies
\begin{equation}\label{brstau}d\tau\wedge dB\equiv 8\pi^2 dr\wedge ds,\quad\forall (\tau, B)\in \Sigma_{n}.\end{equation}
\end{theorem}

\begin{Remark}\label{rmk2-4}
By Theorem \ref{thm-rsbtau} and by taking the open neighborhoods $U, V$ smaller if necessary, it follows from the inverse mapping theorem that the inverse map of the map given in \eqref{eq-UV}
\begin{equation}\label{eq-UV-1} \mathbb{C}^2\setminus\tfrac12\mathbb{Z}^2\supset V\ni (r,s)\mapsto (\tau, B)\in U\subset\Sigma_n\end{equation}
is well-defined and also holomorphic. Denote $\tau_r=\frac{\partial \tau}{\partial r}$, $r_{\tau}=\frac{\partial r}{\partial\tau}$ and so on. Then
\[\begin{pmatrix}\tau_r & \tau_s\\ B_r& B_s\end{pmatrix}\begin{pmatrix}r_{\tau} & r_B\\ s_{\tau}& s_B\end{pmatrix}=\begin{pmatrix}1 & 0\\ 0& 1\end{pmatrix}.\]
Since \eqref{brstau} implies $\tau_rB_s-\tau_sB_r\equiv 8\pi^2$, we get $(\tau_r, \tau_s)\neq (0,0)$ and
\[\begin{pmatrix}B_s & -\tau_s\\ -B_r& \tau_r\end{pmatrix}={8\pi^2}\begin{pmatrix}r_{\tau} & r_B\\ s_{\tau}& s_B\end{pmatrix},\]
namely
\begin{equation}\label{eq:universal}
\tau_s=-8\pi^2r_B,\quad \tau_r=8\pi^2s_B,\quad (\tau_r, \tau_s)\neq (0,0).
\end{equation}
This \eqref{eq:universal} will be used to prove \eqref{eq:nontri-degree} in Section \ref{sec 4}.
\end{Remark}

\begin{lemma}\label{lem:normalization}
Given $B\in\mathbb C$ such that $\ell_n(B)\neq 0$. Then there exists a constant $\kappa=\kappa(B)\neq 0$ such that
\begin{align}
 y_\pm(z)&=\varphi(z;B)\pm\kappa(B)\psi(z;B).\label{eq:kappa}
\end{align}
Consequently, we define
\begin{equation}\label{eq:XiW}
 \Xi(z)=\Xi(z;B):=y_+(z)y_-(z),\quad
 \WW:=y_-(z)y_+'(z)-y_-'(z)y_+(z),
\end{equation}
then $\Xi(z)=z^{-2n}+\cdots$ is an even elliptic function and
\begin{align}
 \WW=2(2n+1)\kappa(B)\ne0.\label{eq:Wronskianvalue}
\end{align}
\end{lemma}

\begin{proof}
Lemma \ref{lem:frobenius} shows that $\varphi(z;B)=z^{-n}+\cdots$ and $\psi(z;B)=z^{n+1}+\cdots$ are linearly independent solutions of the Lam\'{e} equation \eqref{eq:lame}. It follow from \eqref{ff2-1} that there is a constant $\kappa=\kappa(B)\in\mathbb C$ such that $y_+(z)=\varphi(z;B)+\kappa(B)\psi(z;B)$. Since
\[
 \varphi_B(-z)=(-1)^n\varphi_B(z),\qquad
 \psi_B(-z)=(-1)^{n+1}\psi_B(z),
\]
we see from \eqref{eq:reflection} that $y_-(z)=\varphi(z;B)-\kappa(B)\psi(z;B)$. Since $\ell_n(B)\neq 0$ implies that $y_+(z)$ and $y_-(z)$ are linearly independent, we obtain $\kappa(B)\neq 0$.

Clearly \eqref{eq:reflection} and \eqref{mono1}-\eqref{mono2} imply that $\Xi(z)=z^{-2n}+\cdots$ is an even elliptic function. On the other hand, the Wronskian $\varphi\psi'-\varphi'\psi$ is a constant and then equals $2n+1$ by using $\varphi(z)=z^{-n}+\cdots$ and $\psi(z)=z^{n+1}+\cdots$ near $z=0$. This, together with \eqref{eq:kappa}, gives \eqref{eq:Wronskianvalue}.
\end{proof}

\subsection{Connection between Lam\'{e} equation and multiple Green function}
The Green function $G$ on $E_{\tau}$ can be expressed explicitly in terms of elliptic
functions. In particular, by using the complex variable $z$, it was proved in \cite{LW} that
\begin{align}
\label{G_z}-4\pi G_z(z)=-4\pi \frac{\partial G}{\partial z}(z)&=\zeta(z)-r\eta_{1}
-s\eta_{2}
=\zeta(z)-\alpha z-\beta\bar{z},
\end{align}
where $(r,s)$ is defined by  \begin{equation}\label{realrs}z=r+s\tau\quad\text{ with }\quad r,s\in \mathbb{R},\end{equation} 
and
\begin{equation}\label{eq:alphabeta}\alpha=\eta_1-\frac{\pi}{b},\quad \beta=\frac{\pi}{b},\quad\text{where }b:=\operatorname{Im}\tau.\end{equation}

By (\ref{G_z}), the critical point
equations (\ref{poho}) can be translated into the following equivalent
system:
\begin{equation}
\sum_{k\not =j}(\zeta(a_{j}-a_{k})+\zeta(a_{k})-\zeta(a_{j}))=0,\quad\forall1\leq j\leq n\label{ff3-1}%
\end{equation}
(there are only $n-1$ independent equations), and
\begin{equation}
\sum_{j=1}^{n}G_z(a_{j})=0,\label{eq:F}%
\end{equation}
subject to the constraint $\boldsymbol{a}\in \mathcal U_\tau^n$, i.e.,
\[
a_{j}\not =0,\quad a_{j}\not =a_{k}\quad \text{in}\;E_{\tau},\quad \forall
j\not =k.
\]
See \cite{CLW} for the proof. Recalling the definition \eqref{hyper} of $Y_n$, we conclude that $\boldsymbol{a}$ is a critical point of 
$G_n$ if and only if $\boldsymbol a\in Y_n$ and satisfies \eqref{eq:F}. On the other hand, since $G(z)=G(-z)$, \eqref{eq:F} holds automatically for any branch point of $Y_n$. Therefore the above arguments yield the following result.

\begin{theorem}\cite{CLW}\label{thm-CLW}
\begin{itemize}
\item[(1)] $\boldsymbol a$ is a trivial critical point of $G_n$ if and only if $\boldsymbol a$ is a branch point of $Y_n$ (i.e., $\ell_n(B_{\boldsymbol a})=0$). Consequently, there is a one-to-one correspondence between trivial critical points of $G_n$ and distinct roots of $\ell_n$.
\item[(2)] $\boldsymbol{a}$ is a nontrivial critical point of $G_n$ if and only if $\boldsymbol{a}\in Y_n$ is not a branch point (i.e., $\ell_n(B_{\boldsymbol{a}})\neq 0$) and satisfies \eqref{eq:F}. In particular, for a nontrivial critical point $\boldsymbol{a}$ of $G_n$, we have
\[\{a_1,\cdots, a_n\}\cap\{-a_1,\cdots, -a_n\}=\emptyset \quad\text{in }E_{\tau}.\]
\end{itemize}
\end{theorem}

Note that Theorem \ref{Thm-B} is a consequence of Theorem \ref{thm-CLW}-(1) and Theorem \ref{CLWWW}-(4).

Suppose $\boldsymbol{a}=\{a_1,\cdots,a_n\}$ is a nontrivial critical point of $G_n$, and
write
\begin{equation}\label{sumaj} \sum_{j=1}^na_j=r+s\tau,\quad\text{ with }\quad r,s\in\mathbb{R}.
\end{equation}
Then it follows from (\ref{G_z})-\eqref{realrs} that \eqref{eq:F} is equivalent to 
\[\sum_{j=1}^n\zeta(a_j)=r\eta_1+s\eta_2,\]
which is precisely the second equation of \eqref{de-rs}. Then it follows from Theorem \ref{CLWWW} that $(r,s)\in\mathbb R^2\setminus\frac{1}{2}\mathbb{Z}^2$ is the monodromy data of the Lam\'{e} equation \eqref{eq:lame} with $B=B_{\boldsymbol a}$. Note from $(r,s)\in\mathbb R^2\setminus\frac{1}{2}\mathbb{Z}^2$ and \eqref{mono1}-\eqref{mono2} that the monodromy of  \eqref{eq:lame} with $B=B_{\boldsymbol a}$ is unitary. The converse statement also holds as shown in the following result.

\begin{theorem}\cite{CLW}\label{thm-clw2}
\begin{itemize}
\item[(1)] Suppose $\boldsymbol{a}=\{a_1,\cdots,a_n\}$ is a nontrivial critical point of $G_n$, and define $(r,s)$ by \eqref{sumaj}.
Then $(r,s)\in\mathbb{R}^2\setminus\frac{1}{2}\mathbb{Z}^2$ is the monodromy data of  the Lam\'{e} equation \eqref{eq:lame} with $B=B_{\boldsymbol a}$. Consequently, by defining the meromphic function 
\begin{equation}\label{eq:devef}
f(z):=\frac{y_+(z)}{y_-(z)}=(-1)^ne^{2c(\boldsymbol a)z}\prod_{j=1}^n\frac{\sigma(z-a_j)}{\sigma(z+a_j)},
\end{equation}
then the $u_{\lambda}(z)$ defined by \eqref{eq:scaling} via this $f(z)$ is precisely the one-parameter family of solutions of the curvature equation \eqref{mfe} corresponding to this pair $\pm \boldsymbol a$.
\item[(2)]  Suppose $(r,s)\in\mathbb{R}^2\setminus\frac{1}{2}\mathbb{Z}^2$ is the monodromy data of  the Lam\'{e} equation \eqref{eq:lame} for some $B\in\mathbb C$. Then $G_n$ has a unique nontrivial critical point $\boldsymbol{a}=\{a_1,\cdots,a_n\}$ satisfying $\sum_{j=1}^na_j=r+s\tau$ and $B_{\boldsymbol a}=B$.
\end{itemize} 
\end{theorem}

\section{Finite-part integrals and a strict inequality}
\label{sec:finite}

The purpose of this section is to compute the finite-part integrals for some functions related to the solutions $y_{\pm}$ of the Lam\'{e} equation, and to prove a strict inequality; see Proposition \ref{prop:negative}. This inequality will play crucial roles in the proofs of both Theorem \ref{main-thm-1} and Theorem \ref{thm:Jacobi} in the next two sections.

Fix any $B\in\mathbb{C}$ such that $\ell_n(B)\neq 0$ and the corresponding $(r,s)\in\mathbb{R}^2\setminus\frac{1}{2}\mathbb{Z}^2$. Then \eqref{mono1}-\eqref{mono2} imply that $|y_+|$ and $|y_-|$ are both doubly periodic, so the real and doubly periodic function
\begin{equation}\label{eq:realW}
 W(z;B):=|y_+(z)|^2-|y_-(z)|^2
\end{equation}
is well-defined.

\begin{lemma}
For $B\in\mathbb{C}$ such that $\ell_n(B)\neq 0$ and the corresponding $(r,s)\in\mathbb{R}^2\setminus\frac{1}{2}\mathbb{Z}^2$, we have
\begin{align}
 0<\int_{E_\tau}W(z;B)^2\dd A<\infty.\label{eq:Wpositive}
\end{align}
\end{lemma}

\begin{proof}
By \eqref{eq:kappa} and \eqref{eq:frob}, we have
$$W(z;B)=4\re(\kappa(B)\psi(z;B)\overline{\varphi(z;B)})=O(|z|)\quad\text{near }z=0,$$
so $W(z;B)$ is bounded in $E_{\tau}$, i.e., $0\leq\int_{E_\tau}W(z;B)^2\dd A<\infty$. If 
$\int_{E_\tau}W(z;B)^2\dd A=0$, then $W(z;B)\equiv 0$, i.e.,
$|y_+/y_-|\equiv 1$. The open mapping theorem
forces the meromorphic function $y_+/y_-$ to be a constant, a contradiction with the fact that $y_+$ and $y_-$ are linearly independent because $\ell_n(B)\neq 0$.\end{proof}

\subsection{Finite-part integrals} In this preliminary section, we recall the notion of finite-part integrals and study some basic properties.
Choose a fundamental parallelogram of the torus $E_{\tau}$ by $$P:=\big\{t_1+t_2\tau\;|\; t_1, t_2\in (-1/2, 1/2)\big\},$$ 
and let
$$P_\epsilon:=P\setminus\{z\in\mathbb C\;|\;|z|<\epsilon\}\quad\text{for }\epsilon>0 \;\text{small}.$$

\begin{definition}\label{def:fp}
Suppose $g(z)$ is meromorphic with poles lying in the lattice $\Lambda_\tau$, $|g(z)|$ is doubly periodic, and
\begin{equation}\label{eq:gLaurent}
 g(z)=\sum_{k=-m}^\infty g_kz^k,\qquad m\geq 2,\qquad g_{-1}=0.
\end{equation}
Define the finite-part integral as
\begin{equation}\label{eq:fpdef}
 \FP\int_{E_\tau}|g|^2\dd A
: =\lim_{\epsilon\downarrow0}\left[
 \int_{P_\epsilon}|g|^2\dd A-D_g(\epsilon)\right],
\end{equation}
where 
\[D_g(\epsilon):=-\pi\sum_{k=-m}^{-2}\frac{|g_k|^2}{k+1}
                   \epsilon^{2k+2}.\]
\end{definition}

We need to prove that \eqref{eq:fpdef} is well-defined. Indeed, fix a small $\rho>0$ such that $\{z\in\mathbb C\;|\;|z|<2\rho\}\subset P$ and
$\sum_{k=0}^\infty
   \frac{|g_k|^2}{k+1}
       \rho^{2k+2}$ converges, then \eqref{eq:gLaurent} and $\int_0^{2\pi} e^{ik\theta}d\theta=0$ for any $k\in\mathbb Z\setminus\{0\}$ give
\begin{align}
 \int_{\epsilon<|z|<\rho}|g|^2\dd A
 &=2\pi\sum_{\substack{k=-m\\k\ne-1}}^\infty
     |g_k|^2\int_\epsilon^\rho r^{2k+1}\dd r
 =\pi\sum_{\substack{k=-m\\k\ne-1}}^\infty
   \frac{|g_k|^2}{k+1}
       \bigl(\rho^{2k+2}-\epsilon^{2k+2}\bigr)\notag\\
        &=\pi\sum_{\substack{k=-m\\k\ne-1}}^\infty
   \frac{|g_k|^2}{k+1}
       \rho^{2k+2}+D_g(\epsilon)-\pi\sum_{k=0}^\infty
   \frac{|g_k|^2}{k+1}\epsilon^{2k+2},
 \label{eq:radial}
\end{align}
so
$$\FP\int_{E_\tau}|g|^2\dd A=\int_{P_\rho}|g|^2\dd A+\pi\sum_{\substack{k=-m\\k\ne-1}}^\infty
   \frac{|g_k|^2}{k+1}
       \rho^{2k+2}\neq \infty$$
is well-defined.

Note that for any $k,m\in\mathbb Z$, the counterclockwise orientation gives
\begin{align}
 \int_{|z|=\epsilon}z^k\bar z^m\dd\bar z
 &=-i\epsilon^{k+m+1}
       \int_0^{2\pi}e^{i(k-m-1)\theta}\dd\theta\notag\\
 &=\begin{cases}
 -2\pi i\epsilon^{2k},&k=m+1,\\
 0,&k\ne m+1.
 \end{cases}\label{eq:circle}
\end{align}
Let 
$$ K(z):=K_0+\sum_{\substack{k=-m\\k\ne-1}}^\infty
                   \frac{g_k}{k+1}z^{k+1}$$
                   be a primitive of $g(z)$, where $K_0\in\mathbb C$ can be arbitrary chosen. Then by \eqref{eq:circle}, a direct computation gives
\begin{align}
 -\frac i2\int_{|z|=\epsilon}K\bar g\dd\bar z
 &=-\pi\sum_{\substack{k=-m\\k\ne-1}}^\infty
                  \frac{|g_k|^2}{k+1}\epsilon^{2k+2}
 =D_g(\epsilon)
   -\pi\sum_{k=0}^\infty\frac{|g_k|^2}{k+1}\epsilon^{2k+2}\notag\\
 & =D_g(\epsilon)+O(\epsilon^2), \quad\text{for $\epsilon>0$ small,}
 \label{eq:boundary-remainder}
\end{align}
where the term containing $K_0$ vanishes because $g_{-1}=0$, 
and the last series is $O(\epsilon^2)$ by the convergence of the regular
part of $g(z)$ on any fixed smaller disk.

\subsection{A strict inequality} Now we apply the finite-part integrals to $y_{\pm}$ and $\Xi=y_+y_-$. Since we assume $(r,s)\in\mathbb R^2\setminus\frac12\mathbb Z^2$, so $|y_{\pm}|$ and $\Xi$ are doubly periodic.
Furthermore, \eqref{eq:frob} and \eqref{eq:kappa} imply that for each of the meromorphic functions $y_{\pm}^2$ and $\Xi=y_+y_-$, the Laurent expansion at $z=0$ does not contain the term $z^{-1}$, i.e.,
\begin{equation}\label{eq:zerores}
 \Res_{z=0}y_+^2=\Res_{z=0}y_-^2
 =\Res_{z=0}\Xi=0,
\end{equation}
so $ \FP\int_{E_\tau}|y_\pm|^4\dd A$ and $\FP\int_{E_\tau}|\Xi|^2\dd A$ are all well-defined. 
Set
\begin{equation}\label{eq:periods}
 I_j:=\int_{z_0}^{z_0+\omega_j}\Xi(z; B)\dd z,
 \qquad j=1,2,
\end{equation}
where the paths avoid lattice points. Remark that the integral $I_j$ is independent of the choices of $z_0$ and lattice-avoiding paths.
Indeed, since $\Xi(z)=\Xi(z;B)$ is doubly periodic,  \eqref{eq:zerores} implies  $\Res_{z=\omega}\Xi=0$ for any $\omega\in\Lambda_{\tau}$, so the integration of $\Xi(z)$ around every pole gives
zero. Consequently, a single-valued meromorphic primitive $K(z)$ of $\Xi(z)$
exists on $\C$. Then
\begin{equation}\label{eq:addperiods}
 K(z+\omega_j)=K(z)+I_j,\quad \forall z\in\C, \quad j=1,2.
\end{equation}

The purpose of this section is to prove the following important inequality.

\begin{proposition}\label{prop:negative}
For $B\in\mathbb{C}$ such that $\ell_n(B)\neq 0$ and the corresponding $(r,s)\in\mathbb{R}^2\setminus\frac{1}{2}\mathbb{Z}^2$, we have
\begin{equation}\label{eq:negativeperiod}
\im(\overline{I_1}I_2)=\FP\int_{E_\tau}|\Xi(z;B)|^2\dd A
 =-\frac12\int_{E_\tau}W(z;B)^2\dd A<0.
\end{equation}
\end{proposition}

To prove Proposition \ref{prop:negative}, first we need the following result.

\begin{lemma}\label{lem:twisted}
We have
\begin{equation}\label{eq:fourthzero}
 \FP\int_{E_\tau}|y_+|^4\dd A
 =\FP\int_{E_\tau}|y_-|^4\dd A=0.
\end{equation}
\end{lemma}

\begin{proof}
Let $g=y_\pm^2$ and $$\mu_k=\begin{cases}e^{\mp 4\pi i s}\quad\text{if }k=1,\\
e^{\pm 4\pi i r}\quad\text{if }k=2,\end{cases}$$ 
then \eqref{mono1}-\eqref{mono2} imply
$$g(z+\omega_k)=\mu_kg(z),\quad k=1,2.$$
Together with
\eqref{eq:zerores}, we have $\Res_{z=\omega}g=0$ for any $\omega\in\Lambda_{\tau}$,  hence $g(z)$ has a single-valued meromorphic primitive $K(z)$
on $\C$. Consequently, there are constants $d_k$ such that
\begin{equation}\label{eq:cocycle}
 K(z+\omega_k)-\mu_kK(z)=d_k,\quad k=1,2.
\end{equation}
Computing $K(z+\omega_1+\omega_2)$ in both
orders gives $\mu_2d_1+d_2
 =\mu_1d_2+d_1$, i.e.,
\begin{align}
 (1-\mu_2)d_1=(1-\mu_1)d_2.\label{eq:cocyclecompat}
\end{align}
Since $(r,s)\in\mathbb R^2\setminus\frac12\mathbb Z^2$, we have that one of $\{\mu_1,\mu_2\}$ can not be $1$, say $\mu_1\neq 1$ for example.
Define
\[
 \widetilde K(z):=K(z)-\frac{d_{1}}{1-\mu_1}.
\]
Then $\widetilde K'=g$, and \eqref{eq:cocyclecompat} yields
\begin{align}
 \widetilde K(z+\omega_k)-\mu_k\widetilde K(z)
 &=d_k
   -(1-\mu_k)\frac{d_{1}}{1-\mu_1}=0,\quad k=1,2.\label{eq:twistedprimitive}
\end{align}
Together with $\mu_k\overline{\mu_k}=1$, we see that the function $\widetilde K\bar g$ is doubly periodic.
Away from the poles,
\begin{equation}\label{eq:exactnorm}
 \partial_z(\widetilde K\bar g)=\widetilde K'\bar g=|g|^2.
\end{equation}
Note that for an analytic function $H(z)$ on a domain $\Omega\subset C$ with piecewise smooth boundary, the Stokes formula gives
\begin{equation}\label{eq:Stokes}
 \int_{\Omega}\partial_z H\dd A
 =\frac i2\int_{\partial\Omega}H\dd\bar z.
\end{equation} 
By \eqref{eq:exactnorm}-\eqref{eq:Stokes},
\begin{equation}\label{eq:twistedstokes}
 \int_{P_\epsilon}|g|^2\dd A
=\frac i2\int_{\partial P_\epsilon}\widetilde K\bar g\dd\bar z =\frac i2\int_{\partial P}\widetilde K\bar g\dd\bar z
  -\frac i2\int_{|z|=\epsilon}\widetilde K\bar g\dd\bar z.
\end{equation}
Since $\widetilde K\bar g$ is doubly periodic, we have $\int_{\partial P}\widetilde K\bar g\dd\bar z=0$. Together with \eqref{eq:boundary-remainder}, we get
\[
 \int_{P_\epsilon}|g|^2\dd A-D_g(\epsilon)=O(\epsilon^2).
\]
Taking the limit in \eqref{eq:fpdef} proves \eqref{eq:fourthzero}.
\end{proof}

\begin{proof}[Proof of Proposition \ref{prop:negative}]
First, we prove that
\begin{equation}\label{eq:bilinear}
 \FP\int_{E_\tau}|\Xi(z;B)|^2\dd A
 =\im(\overline{I_1}I_2).
\end{equation}

Recall the primitive $K$ of $\Xi$ in \eqref{eq:addperiods}.
Apply \eqref{eq:Stokes} to $K\overline{\Xi}$ leads to
\begin{equation}\label{eq:twistedstokes-1}
 \int_{P_\epsilon}|\Xi|^2\dd A
=\frac i2\int_{\partial P_\epsilon}K\overline{\Xi}\dd\bar z =\frac i2\int_{\partial P}K\overline{\Xi}\dd\bar z
  -\frac i2\int_{|z|=\epsilon}K\overline{\Xi}\dd\bar z.
\end{equation}
For the outer boundary, take $z_0=-\frac12-\frac12\tau$ and
orient $\partial P$ by
$z_0\to z_0+1\to z_0+1+\tau\to z_0+\tau\to z_0$.
Then it follows from \eqref{eq:addperiods} that
\begin{align*}
\int_{z_0}^{z_0+1}K(z)\overline{\Xi(z)}\dd\bar z
 -\int_{z_0}^{z_0+1}(K(z)+I_2)\overline{\Xi(z)}\dd\bar z
=-I_2\overline{I_1},
\end{align*}
\begin{align*}
 \int_{z_0}^{z_0+\tau}(K(z)+I_1)\overline{\Xi(z)}\dd\bar z
 -\int_{z_0}^{z_0+\tau}K(z)\overline{\Xi(z)}\dd\bar z=I_1\overline{I_2},
\end{align*}
which implies
$$\frac i2\int_{\partial P}K\overline{\Xi}\dd\bar z=\frac i2(I_1\overline{I_2}-I_2\overline{I_1})=\im(\overline{I_1}I_2).$$
Together with \eqref{eq:boundary-remainder}, we conclude from \eqref{eq:twistedstokes-1} that \eqref{eq:bilinear} holds.

To prove \eqref{eq:negativeperiod}, we 
note from \eqref{eq:realW} that
\begin{align}
 |y_+|^4+|y_-|^4-2|\Xi|^2
 &=(|y_+|^2-|y_-|^2)^2=W(z;B)^2.\label{eq:polarization}
\end{align}
Since \eqref{eq:frob} and \eqref{eq:kappa} give
\begin{align}
 &y_\pm^2=\varphi^2\pm2\kappa\varphi\psi
                      +\kappa^2\psi^2,\quad
 \Xi=\varphi^2-\kappa^2\psi^2,\notag\\
& \varphi\psi=O(z),\quad \psi^2=O(z^{2n+2})\quad\text{near }z=0,\label{eq:commonprincipal}
\end{align}
it turns out that the three meromorphic functions $y_+^2,y_-^2,\Xi$ have the
same principal part at $0$, so
$$D_{y_+^2}(\epsilon)=D_{y_-^2}(\epsilon)=D_{\Xi}(\epsilon).$$
Then integrating
\eqref{eq:polarization} over $P_\epsilon$ gives
\begin{align*}
 \int_{P_\epsilon}W(z;B)^2\dd A
 &=\left(\int_{P_\epsilon}|y_+|^4\dd A-D_{y_+^2}(\epsilon)\right)
  +\left(\int_{P_\epsilon}|y_-|^4\dd A-D_{y_-^2}(\epsilon)\right)\\
 &\quad-2\left(\int_{P_\epsilon}|\Xi|^2\dd A-D_{\Xi}(\epsilon)\right).
\end{align*}
Letting $\epsilon\downarrow0$, it follows from \eqref{eq:fourthzero} that $$\FP\int_{E_\tau}|\Xi(z;B)|^2\dd A
 =-\frac12\int_{E_\tau}W(z;B)^2\dd A.$$ This proves \eqref{eq:negativeperiod} by using \eqref{eq:Wpositive} and \eqref{eq:bilinear}. 
\end{proof}

\section{Proof of Theorem \ref{main-thm-1}}
\label{sec 4}

This section is devoted to the proof of Theorem \ref{main-thm-1}. The two key steps is to prove \eqref{eq:nontri-degree} and to compute the total topological degree of $\nabla G_n=0$.

\subsection{Non-degeneracy of nontrivial critical points}
Fix any $\tau$ and
define
\[\Sigma_{n}(\tau):=\{B\in\mathbb{C} :  (\tau, B)\in\Sigma_n\}=\{B\in\mathbb{C} :  \ell_n(B;\tau)\neq 0\}.\]
Recall the hyperelliptic curve $Y_n$ defined in \eqref{hyper}.
It follows from Theorem \ref{CLWWW} that $B\in \Sigma_{n}(\tau)$ can be a local coordinate for the hyperelliptic curve $Y_{\mathbf{n}}$, namely for any $j$, $a_j=a_j(B)$ is locally holomorphic for $B\in \Sigma_{n}(\tau)$. From here and \eqref{eq:rs-ex}, i.e.,\begin{equation*}
-2\pi i s=\sum_{j=1}^{n}\zeta(a_{j}(B))-\eta_1\sum_{j=1}^{n}a_j(B),\quad 2\pi i r=\tau \sum_{j=1}^{n}\zeta(a_{j}(B))-\eta_2\sum_{j=1}^{n}a_j(B),
\end{equation*} we see that the corresponding $(r,s)=(r(B), s(B))$ is locally holomorphic for $B\in \Sigma_{n}(\tau)$. Of course, this fact can also follow directly from Theorem \ref{thm-rsbtau} and $(r(B), s(B))=(r(\tau,B), s(\tau, B))$ with fixed $\tau$, where $(r(\tau,B), s(\tau, B))$ is the map in Theorem \ref{thm-rsbtau}. 
Denote
\begin{equation}\label{eq:chij}\chi_j(B):=\begin{cases}e^{-2\pi i s(B)}\quad\text{if }j=1,\\
e^{2\pi i r(B)}\quad\text{if }j=2.\end{cases}\end{equation}
Then \eqref{mono1}-\eqref{mono2} imply
\begin{align}\label{eq:translaw}y_+(z+\omega_j)=\chi_j(B)y_+(z),\quad
y_-(z+\omega_j)=\chi_j(B)^{-1}y_-(z),\quad j=1,2.\end{align}
Note from \eqref{ff2-1} that $y_+(z)=y_{\boldsymbol{a}(B)}(z)$ is also meromorphic in $B\in \Sigma_n(\tau)$. Recall the Wronskian $\WW=y_-y_+'-y_-'y_+\neq 0$ in \eqref{eq:XiW}-\eqref{eq:Wronskianvalue}.

\begin{lemma}\label{lem:variation}
Noting that the quantity $I_j$ in \eqref{eq:periods} is actually well-defined for any $B\in \Sigma_n(\tau)$, we have
\begin{equation}\label{eq:variation}
 I_j=\begin{cases}-2\pi i\WW s_{B},\quad\text{if }j=1,\\ 2\pi i\WW r_{B},\quad\text{if }j=2,\end{cases}
\end{equation}
where $r_B=\frac{\partial r}{\partial B}$ and $s_B=\frac{\partial s}{\partial B}$.
\end{lemma}

\begin{proof}
Set $V:=n(n+1)\wp+B$ and $v:=\partial_B y_+$. Differentiation
of the Lam\'{e} equation $y_{\pm}''=Vy_{\pm}$ with respect to $B$ gives
\begin{equation}\label{eq:inhomogeneous}
 v''=Vv+y_+,\qquad y_-''=Vy_-.
\end{equation}
Thus the meromorphic function $H:=y_-v'-y_-'v$ satisfies
\begin{align}
 H'=y_-'v'+y_-v''-y_-''v-y_-'v'
   =y_-(Vv+y_+)-Vy_-v=\Xi,\label{eq:concomitant}
\end{align}
so $I_j=H(z+\omega_j)-H(z)$ for any $z$.

On the other hand, \eqref{eq:translaw} gives
\begin{align*}
 v(z+\omega_j)=\chi_jv(z)+\partial_B\chi_{j}y_+(z),&\quad
 v'(z+\omega_j)=\chi_jv'(z)+\partial_B\chi_{j}y_+'(z),\\
 y_-(z+\omega_j)=\chi_j^{-1}y_-(z),&\quad
 y_-'(z+\omega_j)=\chi_j^{-1}y_-'(z).
\end{align*}
Substituting these formulas into $H=y_-v'-y_-'v$ yields
$$
 I_j=H(z+\omega_j)-H(z)
 =\frac{\partial_B\chi_{j}}{\chi_j}
        (y_-y_+'-y_-'y_+)=\WW\,\partial_B\log\chi_j.
$$
This proves \eqref{eq:variation} by using \eqref{eq:chij}.
\end{proof}

Now we recall the following important formula for the Hessian of nontrivial critical points from \cite{CFL-Crelle}.

\begin{theorem}\cite{CFL-Crelle} \label{DD-conj}
Fix any $n \geq 1$, $\tau\in\mathbb H$ and write $b=\im \tau>0$. Suppose $\boldsymbol{a} = \{a_1, \cdots, a_n\}$ is a nontrivial critical point of $G_n$, then there is a constant $c_{\boldsymbol{a}} >0 $ such that
\begin{equation}\label{eq-ex-non}
\det D^2 G_n(\boldsymbol{a}) = \frac{(-1)^n n^2c_{\boldsymbol{a}}}{4(2\pi)^{2n+2} b}  \operatorname{Im}
\big((\overline{\tau_r}\tau_s)(r_{\boldsymbol{a}}, s_{\boldsymbol{a}})\big).
\end{equation}
Here $(r_{\boldsymbol{a}}, s_{\boldsymbol{a}})\in\mathbb{R}^2\setminus\frac12\mathbb{Z}^2$ is determined by $\sum_{j=1}^na_j=r_{\boldsymbol{a}}+s_{\boldsymbol{a}}\tau$, and $\tau(r,s)$ is the holomorphic function satisfying $\tau(r_{\boldsymbol{a}}, s_{\boldsymbol{a}})=\tau$ that is introduced in Remark \ref{rmk2-4}, with $\tau_r:=\frac{\partial \tau}{\partial r}$, $\tau_s:=\frac{\partial \tau}{\partial s}$ and $(\tau_r, \tau_s)\neq (0,0)$.
\end{theorem}

Now we can prove the first key result of this section.
\begin{theorem} \label{thm-nondegenerate}
Fix any $n\geq 1$, $\tau\in\mathbb H$ and write $b=\im \tau>0$. Suppose $\boldsymbol{a} = \{a_1, \cdots, a_n\}$ is a nontrivial critical point of $G_n$, and let $B=B_{\boldsymbol a}$. Then
\begin{equation}\label{eq-ex-non2}
(-1)^{n+1}\det D^2 G_n(\boldsymbol{a}) = \frac{n^2c_{\boldsymbol{a}}}{2(2\pi)^{2n} b|\WW|^2} \int_{E_\tau}W(z;B)^2\dd A>0.
\end{equation}
\end{theorem}

\begin{proof}
Since $B=B_{\boldsymbol a}$, it follows from Theorems \ref{thm-CLW}-\ref{thm-clw2} that $\ell_n(B)\neq 0$ and the corresponding $(r,s)=(r_{\boldsymbol{a}}, s_{\boldsymbol{a}})\in\mathbb{R}^2\setminus\frac12\mathbb{Z}^2$, so those results in Section \ref{sec:finite} apply. Note from \eqref{eq:universal} and \eqref{eq:variation} that
$$\overline{\tau_r}\tau_s=-64\pi^4r_{B}\overline{s_B}=\frac{16\pi^2}{|\WW|^2}\overline{I_1}I_2.$$
From here, \eqref{eq:negativeperiod} and \eqref{eq-ex-non}, we get \eqref{eq-ex-non2}.
\end{proof}

\subsection{Topological degree}
In this section, we want to show that the total topological degree of $\nabla G_n$ is always $(-1)^n$. The key step is to prove the following result.

\begin{lemma}
\label{lem:index-boundary}
Let $K\Subset\HH$ be any compact subset. Then
there is $\varepsilon=\varepsilon(K)>0$ such that for any critical point $\boldsymbol{a}=\{a_1,\cdots,a_n\}$ of $G_n(\cdot;\tau)$, we have
\begin{equation}\label{eq:index-uniform-separation}
 \begin{cases}
 d_\tau(a_j,0)\ge\varepsilon&\text{for every }j,\\
 d_\tau(a_j,a_k)\ge\varepsilon&\text{for every }j\ne k,
 \end{cases}\quad \forall \tau\in K.
\end{equation}
Here $d_\tau$ is the distance on the flat torus $E_{\tau}$. 

Consequently, the set $\cup_{\tau\in K}\{\text{critical points of }G_n(\cdot;\tau)\}$ is compact.
\end{lemma}

\begin{proof}
Assume by contradiction that such $\varepsilon$ does not exist. Then
there are $\tau_m\in K$ and a critical point $\boldsymbol{a}^{(m)}=\{a^{(m)}_1,\cdots,a^{(m)}_n\}$ of $G_n(\cdot;\tau_m)$ such that as $m\to+\infty$, either $d_{\tau_m}(a^{(m)}_j,0)\to 0$ for some $j$ or  $d_{\tau_m}(a^{(m)}_j,a^{(m)}_k)\to 0$ for some $j\neq k$.

Up to a subsequence and replacing $a^{(m)}_j$ with some elements in $a^{(m)}_j+\Lambda_{\tau_m}$ if necessary, we may assume
$\tau_m\to \tau\in K$ and $$a^{(m)}_j\to a_j\in \big\{t_1+t_2\tau \;:\; t_1, t_2\in [-\tfrac12,\tfrac12]\big\} \quad\text{in }\;\C$$ for any $1\leq j\leq n$, where $\{a_j\}_{j=1}^n$ satisfies $a_j=a_k$ if $a_j-a_k\in \Lambda_{\tau}$ for some $j\neq k$. Then under our assumption, either $a_j=0$ for some $j$ or $a_j=a_k\neq 0$ for some $j\neq k$.

By \eqref{G_z} and \eqref{poho}, we get
\begin{align}
 0=&n\zeta(a_j^{(m)};\tau_m)
  -\sum_{k\ne j}\zeta(a_j^{(m)}-a_k^{(m)};\tau_m)
  -\alpha_m \sum_k a_k^{(m)}-\beta_m\sum_k \overline{a_k^{(m)}}\nonumber\\
  =&n\zeta(a_j^{(m)};\tau_m)
  -\sum_{k\ne j}\zeta(a_j^{(m)}-a_k^{(m)};\tau_m)+O(1),\quad\forall j.
 \label{eq:index-exact-gradient}
\end{align}

{\bf Case 1.} $a_j=0$ for some $j$. 

Denote $I:=\{j : a_j=0\}$ and $l:=\#I\geq 1$. Note that
\begin{equation}\label{eq:index-zeta-expansion}
 \zeta(z;\tau)=\frac1z+O(|z|^3)
\end{equation}
is uniform for $\tau\in K$ near $z=0$. Thus \eqref{eq:index-exact-gradient} implies
\begin{equation}\label{eq:index-origin-cluster}
0
 =\frac n{a_j^{(m)}}
  -\sum_{\substack{k\in I\\k\ne j}}
        \frac1{a_j^{(m)}-a_k^{(m)}}+O(1),
 \qquad j\in I.
\end{equation}
Note that
\begin{align}
 \sum_{j\in I}a_j^{(m)}
       \sum_{\substack{k\in I\\k\ne j}}
       \frac1{a_j^{(m)}-a_k^{(m)}}
 &=\sum_{\substack{j<k\\j,k\in I}}
   \left(
     \frac{a_j^{(m)}}{a_j^{(m)}-a_k^{(m)}}
    +\frac{a_k^{(m)}}{a_k^{(m)}-a_j^{(m)}}
   \right)\notag\\
 &=\frac{l(l-1)}2.
 \label{eq:index-pair-cancellation}
\end{align}
Multiplying \eqref{eq:index-origin-cluster} by $a_j^{(m)}$ and
summing gives
\begin{align*}
 0=nl-\frac{l(l-1)}2+o(1)=\frac{l(2n-l+1)}2+o(1),
\end{align*}
a contradiction.

{\bf Case 2.} $a_j=a_k=:a\neq 0$ for some $j\neq k$. 

Denote $I:=\{j : a_j=a\}$ and $l:=\#I\geq 2$. Then $b^{(m)}_{j}:=a_j^{(m)}-a\to 0$ for any $j\in I$, and \eqref{eq:index-exact-gradient} implies
\begin{equation}\label{eq:index-away-cluster}
 0
 =-\sum_{\substack{k\in I\\k\ne j}}
        \frac1{b_j^{(m)}-b_k^{(m)}}+O(1),
 \qquad j\in I.
\end{equation}
Multiplying \eqref{eq:index-away-cluster} by $b_j^{(m)}$ and
summing gives
$$
0 =-\frac{l(l-1)}2+o(1),
$$
again a contradiction. This proves \eqref{eq:index-uniform-separation}.

Consequently,
take any $\tau_m\in K$ and any $\boldsymbol{a}^{(m)}$ to be a critical point of $G_n(\cdot;\tau_m)$. Again up to a subsequence, we may assume $\tau_m\to \tau\in K$ and $\boldsymbol{a}^{(m)}\to \boldsymbol{a}=\{a_1,\cdots,a_n\}\subset E_{\tau}$. Then  \eqref{eq:index-uniform-separation} implies that $a_j\neq 0$ and $a_j\neq a_k$ in $E_{\tau}$ for any $j\neq k$, so $\boldsymbol{a}$ is a critical point of $G_n(\cdot;\tau)$. Thus, $\cup_{\tau\in K}\{\text{critical points of }G_n(\cdot;\tau)\}$ is compact. 
\end{proof}

Recall $\mathcal U_{\tau}^n$ defined in \eqref{eq:mU}.
Now we can prove the second key result of this section.

\begin{theorem}\label{thm-degree}
Let $n\in\mathbb N^*$ and $\tau\in \mathbb H$. Then the topological degree
\begin{equation}\label{eq:degreeG}
\deg(\nabla G_n(\cdot;\tau), \mathcal U_{\tau}^n, 0)=(-1)^n.
\end{equation}
\end{theorem}

\begin{proof} 
By Lemma \ref{lem:index-boundary}, the homotopy invariance of topological degrees implies that $\deg(\nabla G_n(\cdot;\tau),$ $\mathcal U_{\tau}^n, 0)$ is the same integral constant for any $\tau\in K$, where $K\Subset\mathbb H$ is any compact subset. Thus, $\deg(\nabla G_n(\cdot;\tau), \mathcal U_{\tau}^n, 0)$ is the same integral constant for any $\tau\in \mathbb H$. Since Theorem \ref{Thm-C} shows that \eqref{eq:degreeG} holds for any $\tau\in i\mathbb R_{>0}$, we obtain \eqref{eq:degreeG} for any $\tau\in \mathbb H$.
\end{proof}

\begin{Remark}
One can prove that the Euler characteristic of $\mathcal U_{\tau}^n$ is $\chi(\mathcal U_{\tau}^n)=(-1)^n$ (We do not provide the proof here because this fact is not needed in this paper). Thus, actually 
$$\deg(\nabla G_n(\cdot;\tau), \mathcal U_{\tau}^n, 0)=\chi(\mathcal U_{\tau}^n).$$
\end{Remark}

\subsection{Upper bound for pairs of nontrivial critical points} Now we are ready to prove Theorem \ref{main-thm-1}.

\begin{proof}[Proof of  Theorem \ref{main-thm-1}]
Denote by $N_n(\tau)$ to be the number of pairs of nontrivial critical points of $G_n(\cdot;\tau)$. We want to prove $N_n(\tau)\leq n$ for any $\tau$. Define
$$\mathcal{O}_n:=\big\{\tau\in \mathbb H\,:\, \text{all critical points of $G_n(\cdot;\tau)$ are non-degenerate}\big\}.$$
Then it was proved in \cite{CFL-Crelle} that $\mathcal{O}_n$ is dense in $\mathbb H$.

{\bf Step 1.} We prove $N_n(\tau)\leq n$ for $\tau\in \mathcal{O}_n$.

Fix $\tau\in \mathcal{O}_n$. Then it follows from Theorem \ref{Thm-B} that $G_n(\cdot;\tau)$ has exactly $2n+1$ trivial critical points. Indeed, suppose the number of trivial critical points of $G_n(\cdot;\tau)$ is less than $2n+1$. Since Theorem \ref{Thm-B} implies the existence of $\tau_k\to \tau$ such that $G_n(\cdot;\tau_k)$ has exactly $2n+1$ trivial critical points, it follows that some trivial critical point of $G_n(\cdot;\tau)$ must be degenerate, a contradiction with $\tau\in \mathcal{O}_n$. Therefore, $G_n(\cdot;\tau)$ has exactly $2n+1$ trivial critical points and they are all non-degenerate, hence each of them contributes degree $(-1)^n$ or $(-1)^{n+1}$. Define
\begin{align*}
T_n(\tau):=\#\left\{\boldsymbol{a}\left \vert
\begin{array}
[c]{l}%
\text{ $\boldsymbol{a}$ is a trivial critical point of $G_n(\cdot;\tau)$}\\
(-1)^{n+1}\det D^2G_n(\boldsymbol{a};\tau)>0
\end{array}
\right.\right\}.
\end{align*}
Since \eqref{eq-ex-non2} shows that each non-trivial critical point contributes degree $(-1)^{n+1}$, it follows from \eqref{eq:degreeG} that
$$(-1)^n=(-1)^{n+1}(2N_n(\tau)+T_n(\tau))+(-1)^n(2n+1-T_n(\tau)),$$
or equivalently,
\begin{equation}\label{eq:N-criteria}
0\leq N_n(\tau)=n-T_n(\tau)\leq n.
\end{equation}
This proves $N_n(\tau)\leq n$ and $T_n(\tau)\leq n$. 

{\bf Step 2.} We prove $N_n(\tau)\leq n$ for $\tau\in \mathbb H$.

Fix any $\tau\in\mathbb H$. Assume by contradiction that $N_n(\tau)\geq n+1$, namely $G_n(\cdot;\tau)$ has at least $n+1$ pairs of nontrivial critical points. Since all these nontrivial critical points are non-degenerate, and $\mathcal{O}_n$ is dense in $\mathbb H$, we see from the implict function theorem that there exist $\mathcal{O}_n\ni\tau_k\to \tau$ such that $G_n(\cdot;\tau_k)$ has also at least $n+1$ pairs of nontrivial critical points, a contradiction with $N_n(\tau_k)\leq n$. This proves $N_n(\tau)\leq n$. The proof is complete.
\end{proof}

\section{The linearized equation}
\label{sec:jacobi}

This section is devoted to the proof of Theorem \ref{thm:Jacobi}. Given a one-parameter family of solutions for the curvature equation \eqref{mfe}:
\begin{equation}\label{eq:scaling2}
 u_\lambda(z)=\log\frac{8e^{2\lambda}|f'(z)|^2}
 {(1+e^{2\lambda}|f(z)|^2)^2},\qquad \lambda\in\mathbb R,
\end{equation}
where, by Theorem \ref{thm-clw2}, the developing map $f(z)$ is given by a nontrivial critical point $\boldsymbol a=\{a_1,\cdots,a_n\}$ of $G_n$ as follows
\begin{equation}\label{eq:devef1}
f(z)=\frac{y_+(z)}{y_-(z)}=(-1)^ne^{2c(\boldsymbol a)z}\prod_{j=1}^n\frac{\sigma(z-a_j)}{\sigma(z+a_j)}.
\end{equation}
Let $B=B_{\boldsymbol a}$,
then inserting \eqref{eq:devef1} into \eqref{eq:scaling2} leads to
\begin{align}
 u_\lambda(z)
 =\log\frac{8|\WW|^2}
 {(e^\lambda|y_+(z)|^2+e^{-\lambda}|y_-(z)|^2)^2},\label{eq:family}
\end{align}
 which implies
\begin{equation}\label{eq:familyreflection}
 u_\lambda(-z)=u_{-\lambda}(z).
\end{equation}
If $u_\lambda(z)$ is even, then \eqref{eq:family} and $|y_+(-z)|=|y_-(z)|$ imply
\[
 (e^\lambda-e^{-\lambda})(|y_+|^2-|y_-|^2)=0,
\]
Since $W(z;B)\not\equiv0$ by \eqref{eq:Wpositive}, this forces
$\lambda=0$. Hence $u_0(z)$ is the unique even function among $u_{\lambda}(z)$.

From now on, we fix any $\lambda\in\mathbb{R}$ and let $\varphi(z)$ be a nontrivial solution of the linearized equation
\begin{equation}\label{eq:linn}
\Delta \varphi+e^{u_{\lambda}}\varphi=0\qquad\text{in }E_{\tau}.
\end{equation}
For convenience, we denote $u(z):=u_{\lambda}(z)$ and
\begin{equation}\label{eq:fscale}
 g(z):=e^\lambda f(z)=e^{\lambda}\frac{y_+(z)}{y_-(z)}=e^\lambda(-1)^ne^{2c(\boldsymbol a)z}\prod_{j=1}^n\frac{\sigma(z-a_j)}{\sigma(z+a_j)},
\end{equation}
i.e.,
\begin{equation}\label{eq:scaling3}
 u(z)=\log\frac{8|g'(z)|^2}
 {(1+|g(z)|^2)^2}.
\end{equation}
Then it follows from  \eqref{eq:scaling2} that
\begin{equation}\label{eq:ZZZ}
Z(z):=\frac12\partial_{\lambda}u_{\lambda}(z)=\frac{1-|e^\lambda f|^2}{1+|e^{\lambda}f|^2}=\frac{1-|g(z)|^2}{1+|g(z)|^2}
\end{equation}
is a solution of the linearized equation \eqref{eq:linn}. Our goal is to prove $\varphi\in \mathbb R Z$.

First, it was proved in \cite[Lemma 4.1]{LW4} that $\varphi_{zz}(z)-u_z(z)\varphi_z(z)$ is a constant, i.e.,
\begin{equation}\label{eq:Qdef}
\varphi_{zz}(z)-u_z(z)\varphi_z(z)\equiv q\in\mathbb{C}.
\end{equation}

To proceed, 
we define the real sphere-valued map
\begin{equation}\label{eq:sphere}
 \mathbf S=(S_1,S_2,S_3)
 :=\frac1{1+|g|^2}\bigl(g+\bar g,-i(g-\bar g),1-|g|^2\bigr),
 \quad\text{i.e.,}\;\; S_3=Z.
\end{equation}
Clearly $S_1^2+S_2^2+S_3^2\equiv 1$.
A direct computation gives
\begin{equation}\label{eq:spherederivative}
 \mathbf S_z=\frac{g'}{(1+|g|^2)^2}
       \bigl(1-\bar g^2,-i(1+\bar g^2),-2\bar g\bigr).
\end{equation}
Note from \eqref{eq:scaling3} that
\begin{equation}\label{eq:uz}
 u_z=\frac{g''}{g'}-\frac{2g'\bar g}{1+|g|^2},
\end{equation}
differentiating \eqref{eq:spherederivative} gives
\begin{equation}\label{eq:sphereequations}
 \mathbf S_{z\bar z}
 =-\frac{2|g'|^2}{(1+|g|^2)^2}\mathbf S=-\frac14e^u\mathbf S,\qquad \mathbf S_{zz}=u_z\mathbf S_z.
\end{equation}
Noting from $\Delta=4\partial_{z\bar z}^2$ that the first formula in \eqref{eq:sphereequations} just says that
 $\Delta S_j+e^{u}S_j=0$ (However, $S_1$ and $S_2$ are not doubly periodic (see \eqref{eqeww} below), i.e., not well-defined on $E_{\tau}$, so they are not solutions of the linearized equation \eqref{eq:linn}),
while the second formula says that the corresponding $(S_j)_{zz}-u_z(S_j)_z=0$ for $j=1,2,3$, compared with \eqref{eq:Qdef}. Our proof will consists of the following steps.
\begin{itemize}
\item[(1).] We use \eqref{eq:negativeperiod} to prove the constant $q=0$ in \eqref{eq:Qdef}.
\item[(2).] We use $q=0$ to prove the existence of $b_j\in\mathbb R$ such that $\varphi=b_1S_1+b_2S_2+b_3S_3$.
\item[(3).] We prove $b_1=b_2=0$ by using the fact that $S_1$ and $S_2$ are not doubly periodic, while $\varphi$ and $S_3$ are doubly periodic.
\end{itemize}

Denote $\overline{\mathbb R}=\mathbb R\cup\{\infty\}$.
Inspired by \cite[Section 3]{KZ}, we define a ${\overline{\mathbb R}}^3$-valued map by
\begin{equation}\label{eq:X}
 \mathbf X=(X_1, X_2, X_3):
 =\varphi\mathbf S
  +4e^{-u}(\varphi_z\mathbf S_{\bar z}
                      +\varphi_{\bar z}\mathbf S_z).
\end{equation}
This $\mathbf X$ will play a crucial role in steps (1)-(2). The relation between $\mathbf X$ and $q$ is established as follows.

\begin{lemma}\label{lem:X}
Recall \eqref{eq:Qdef} that $\varphi_{zz}-u_z\varphi_z=q\in\C$, and recall $\Xi, \WW$ defined in \eqref{eq:XiW}. Then
\begin{equation}\label{eq:Xderivative}
 \mathbf X_z=4q e^{-u}\mathbf S_{\bar z},
 \qquad (X_{3})_z=-\frac q\WW\Xi.
\end{equation}
\end{lemma}

\begin{proof}
Differentiating \eqref{eq:X}, and using \eqref{eq:sphereequations} and $\varphi_{z\bar{z}}=-\frac14 e^u\varphi$, we get
{\allowdisplaybreaks
\begin{align}
 \mathbf X_z
 &=\varphi_z\mathbf S+\varphi\mathbf S_z\notag\\
 &\quad+4e^{-u}\bigl[
  (\varphi_{zz}-u_z\varphi_z)\mathbf S_{\bar z}
  +\varphi_{\bar z}(\mathbf S_{zz}-u_z\mathbf S_z)
  +\varphi_z\mathbf S_{z\bar z}
  +\varphi_{z\bar z}\mathbf S_z\bigr]\notag\\
 &=\varphi_z\mathbf S+\varphi\mathbf S_z
   +4e^{-u}\left[q\mathbf S_{\bar z}
       -\frac14e^u\varphi_z\mathbf S
       -\frac14e^u\varphi\mathbf S_z\right]
 =4q e^{-u}\mathbf S_{\bar z}.\label{eq:Xcalculation}
\end{align}
}%
To prove the second formula, note from \eqref{eq:ZZZ} that
\begin{align*}
 (S_3)_{\bar z}=Z_{\bar z}=-\frac{2g\overline{g'}}{(1+|g|^2)^2},\end{align*}
so
 \begin{align}
 4qe^{-u}(S_3)_{\bar z}
 =4q\frac{(1+|g|^2)^2}{8|g'|^2}
            \frac{-2g\overline{g'}}{(1+|g|^2)^2}
 =-q\frac g{g'}.\label{eq:Zderivative}
\end{align}
On the other hand, $g=e^{\lambda}y_+/y_-$ gives
\begin{equation}\label{eq:fratio}
 g'=e^\lambda\frac\WW{y_-^2}\quad\text{and then}\quad
 \frac g{g'}=\frac{y_+y_-}{\WW}=\frac{\Xi}{\WW}.
\end{equation}
Inserting \eqref{eq:fratio} into \eqref{eq:Zderivative} implies the second formula in \eqref{eq:Xderivative}.
\end{proof}

Now we are ready to prove Theorem \ref{thm:Jacobi}.

\begin{proof}[Proof of Theorem \ref{thm:Jacobi}] 
We divide the proof into three steps.

{\bf Step 1.} We prove $\varphi_{zz}-u_z\varphi_z=q=0$.

Since $\varphi$ and $Z$ are real-valued and doubly periodic, so does
 $X_{3}$ in \eqref{eq:X}. Then it follows from \eqref{eq:Xderivative} that
 $$\dd X_{3}=(X_3)_z\dd z+(X_3)_{\bar z}d{\bar z}=(X_3)_zdz+\overline{(X_3)_zdz}=-2\re\Big(\frac q\WW\Xi \dd z\Big).$$
Recalling the quantity $I_j$ defined in
\eqref{eq:periods}, the integration gives
\begin{align}
 0&=X_{3}(z_0+\omega_j)-X_{3}(z_0)
  =-2\re\int_{z_0}^{z_0+\omega_j}\frac q\WW\Xi\dd z\nonumber\\
   &=-2\re\left(\frac q\WW I_j\right),\quad j=1,2.\label{eq:periodconstraint}
\end{align}
Write $q/\WW=x+iy$ and $I_j=p_j+it_j$ with $x,y,p_j,t_j\in\mathbb R$. Then \eqref{eq:periodconstraint} becomes
\begin{equation*}
 \begin{pmatrix}p_1&-t_1\\p_2&-t_2\end{pmatrix}
 \begin{pmatrix}x\\y\end{pmatrix}=0.
\end{equation*}
Since \eqref{eq:negativeperiod} says that 
$$\det\begin{pmatrix}p_1&-t_1\\p_2&-t_2\end{pmatrix}
 =-\im(\overline{I_1}I_2)>0,$$
we get $x=y=0$, i.e., $q=0$. 

{\bf Step 2.} We show the existence of $b_3\in\mathbb R$ such that $\varphi=b_3S_3=b_3Z$. Thus, $\kk_{u_\lambda}
 =\Span\{\partial_\lambda u_\lambda\}$.

By $q=0$, \eqref{eq:Xderivative} becomes
$\mathbf X_z=0$, so $\mathbf X$ is a constant vector $\mathbf b=(b_1,b_2,b_3)\in\R^3$.
Since $\mathbf S\cdot\mathbf S=S_1^2+S_2^2+S_3^2=1$ implies
$$\mathbf S\cdot\mathbf S_z=
 \mathbf S\cdot\mathbf S_{\bar z}=0,$$
 it follows from \eqref{eq:X} that
\begin{equation}\label{eq:recoverphi}
\varphi=\mathbf b\cdot\mathbf S=\sum_{j=1}^3 b_j S_j.
\end{equation}

We claim that $b_1=b_2=0$. Indeed, by defining
$$\mu_k=\begin{cases}e^{-4\pi i s}\quad\text{if }k=1,\\
e^{4\pi i r}\quad\text{if }k=2,\end{cases}$$ 
then $\overline{\mu_k}=\mu_k^{-1}$, and it follows from \eqref{mono1}-\eqref{mono2} and $g=e^{\lambda}y_+/y_-$ that
\begin{equation}\label{eqeww}g(z+\omega_k)=\mu_kg(z),\quad |g(z+\omega_k)|=|g(z)|,\quad k=1,2.\end{equation}
By $(r,s)\in\mathbb R^2\setminus\frac12\mathbb Z^2$, there is $k\in\{1,2\}$ such that $\mu_k\neq 1$.
Since $\varphi$ and $S_3$ is periodic, so is $b_1S_1+b_2S_2$. Then we see from \eqref{eq:sphere} that
\begin{align*}
b_1(g+\bar g)-ib_2(g-\bar g)=b_1(\mu_kg+\mu_k^{-1}\bar g)-ib_2(\mu_kg-\mu_k^{-1}\bar g),
\end{align*}
or equivalently,
$$(b_1-ib_2)g(z)=\frac{b_1+ib_2}{\mu_k}\overline{ g(z)}.$$
If $(b_1,b_2)\neq (0,0)$, then 
$$\frac{g(z)}{\overline{g(z)}}\equiv\frac{b_1+ib_2}{\mu_k(b_1-ib_2)},\quad\forall z,$$
a contradiction with the fact that $g(z)$ is a non-constant meromorphic function. This proves $b_1=b_2=0$ and then $\varphi=b_3S_3=b_3Z$. Thus, $ \kk_{u_\lambda}
 =\Span\{\partial_\lambda u_\lambda\}$.

{\bf Step 3.} Let $\lambda=0$, i.e., $u_0(z)$ is even. Then we see from \eqref{eq:ZZZ} and $f(-z)=1/f(z)$ that $Z(-z)=-Z(z)$, i.e., $Z(z)$ is odd. Thus, $\kk_{u_0}^{\mathrm{ev}}=\{0\}.$ The proof is complete.
\end{proof}

\section{Sharpness for $n=3$} 
\label{sec n=3}

Throughout this final section, let
\[
 \tau=\frac12+\frac{\sqrt3}{2}i,\quad
 b=\operatorname{Im}\tau=\frac{\sqrt3}{2},\quad
 \beta=\frac\pi b=\frac{2\pi}{\sqrt3},
\]
and recall
\[
 G_3(z_1,z_2,z_3)
 =\sum_{1\le i<j\le3}G(z_i-z_j)-3\sum_{i=1}^3G(z_i).
\]
The purpose of this section is to prove the following result, which shows that the upper bound $3$ is optimal for $G_3$ in Theorem \ref{main-thm-1}.

\begin{theorem}\label{thm-g3tau}
$G_3(\cdot;\tau)$ has exactly $3$ pairs of nontrivial critical points.
\end{theorem}

Recall $\omega_1=1$, $\omega_2=\tau$ and $\omega_3=1+\tau$, we set
\begin{equation}\label{eq:tri3-half-periods}
 \quad h_j={\omega_j/2},\quad
 e_j=\wp(h_j),\quad j=1,2,3.
\end{equation}
Then
\begin{equation}\label{eqhjj}
h_j\equiv\tau^{j-1}h_1\pmod{\Lambda_\tau},\quad j=2,3.
\end{equation}
Recall Theorem \ref{Thm-B} that $G_3(\cdot;\tau)$ has at most seven trivial critical points.
By \eqref{eq:degreeG} and \eqref{eq-ex-non2}, we see that Theorem \ref{thm-g3tau} follows directly from the following result, which says that $G_3(\cdot;\tau)$ has exactly seven trivial critical points, the Hessians of which are all negative, so all trivial critical points contribute the total degree $-7$.

\begin{proposition}\label{prop:tri3-trivial-hessians}
Choose $a_+\in E_\tau$ such that
\begin{equation}\label{eq:tri3-a-definition}
 \wp(a_+)=\frac{-1+3i}{5}e_1,
 \qquad a_-:=\overline{a_+}.
\end{equation}
Then $G_3$ has exactly seven trivial critical points, which are precisely
\begin{align}\label{eq:tri3-seven-configurations}
 &\mathbf h=\{h_1,h_2,h_3\},\\
 &\boldsymbol  a_{j,\pm}
 =\{\tau^{j-1}h_1,\tau^{j-1}a_\pm,-\tau^{j-1}a_\pm\}=\tau^{j-1}\boldsymbol a_{1,\pm},\quad j=1,2,3.\nonumber
\end{align}
Furthermore,
{\allowdisplaybreaks
\begin{align*}
 \det D^2G_3(\mathbf h)
 &=-\frac{729e_1^6}{16\pi^6}<0,\\
 \det D^2G_3(\boldsymbol a_{j,\pm})
 &=-\frac{729e_1^4}{10\pi^6}
       \left(e_1^2-\frac{10\pi^2}{3}\right)<0,
 \qquad j=1,2,3.
\end{align*}
}%
\end{proposition}

\begin{proof} 
Since $\tau\Lambda_\tau=\Lambda_\tau$, $\overline{\Lambda_\tau}=\Lambda_\tau$ and $\tau^6=1$, it is well known that
\begin{equation}\label{eq:tri3-scaling}
 \wp(\tau z)=\tau^{-2}\wp(z),\qquad
 \zeta(\tau z)=\tau^{-1}\zeta(z),\qquad
 \wp(\overline z)=\overline{\wp(z)},
\end{equation}
\begin{equation}\label{e0q:tri3-scaling}
 g_2=0,\quad e_2=\tau^4e_1,\quad e_3=\tau^2 e_1,\quad e_1\in\mathbb R,
 \quad g_3=4e_1^3=4e_j^3,
\end{equation}
and the uniqueness of the Green function
gives
\begin{equation}\label{eqGtau}
 G(\tau z)=G(z),\qquad G(\overline z)=G(z)=G(-z).
\end{equation}
Furthermore, enumerically (see e.g. \cite[p. 31]{CFL}),
\begin{equation}\label{eq:tri3-e-bound}
e_1\approx 1.8775\pi>
       \sqrt{\frac{10}{3}}\pi.
\end{equation}
Recall $\eta_1=2\zeta(1/2)$ and $\eta_2=2\zeta(\tau/2)$. Denote $\eta_3=\eta_1+\eta_2=2\zeta(\frac{1+\tau}{2})$.
Then \eqref{eq:tri3-scaling} and the Legendre
relation $\tau\eta_1-\eta_2=2\pi i$ imply
\begin{equation}\label{eqetav}
\eta_1=\frac{\pi}{b}=\beta,\quad \eta_2=\tau^{-1}\eta_1=\beta\overline\tau,\quad
\eta_3=\beta(1+\overline\tau).
\end{equation}

{\bf Step 1. Classification of the trivial critical points.}

Recalling \eqref{G_z}, we define
\begin{equation}\label{eq:tri3-green-gradient}
 F(z):=-4\pi G_z(z)=\zeta(z)-\beta\overline z.
\end{equation}
Clearly $F(z)$ is odd and $F(h_j)=0$ for $j=1,2,3$.

Since a trivial critical point $\boldsymbol a=\{a_1, a_2, a_3\}$ satisfies $\{a_1, a_2, a_3\}=\{-a_1, -a_2, -a_3\}$ in $E_{\tau}$, and $a_i\neq 0$, $a_i\neq a_j$ for $i\neq j$, we see that every trivial critical point is of the form
\[
 \mathbf h=\{h_1,h_2,h_3\}
 \quad\text{or}\quad
 \{h_j,u,-u\},\quad u\notin E_\tau[2]:=\{0, h_1,h_2,h_3\}.
\]
Conversely, a direct computation via \eqref{poho} shows that $\mathbf h=\{h_1, h_2, h_3\}$ is a trivial critical point of $G_3$.

Now we fix $j\in\{1,2,3\}$ and consider $\{h_j,u,-u\}$. Note that $\nabla_{h_j}G_3(h_j, u,-u)=0$ automatically,
because
$h_j+u\equiv-(h_j-u)\pmod{\Lambda_\tau}$ and $F(h_j)=0$. Furthermore,
the critical equations at $u$ and $-u$ are equivalent by oddness.
Thus $\{h_j,u,-u\}$ is a trivial critical point of $G_3$ if and only if $\nabla_{u}G_3(h_j, u, -u)=0$, or equivalently,
$$ F(u-h_j)+F(2u)-3F(u)=0.$$
By \eqref{eq:tri3-green-gradient} and \eqref{eqetav} which says that $\zeta(h_j)=\beta\overline{ h_j}$, this equation is equivalent to
\begin{equation}\label{eq:tri3-reduced-critical}
\zeta(u-h_j)+\zeta(h_j)+\zeta(2u)-3\zeta(u)=0.
\end{equation}
Since $u\notin E_\tau[2]$, we have $\wp'(u)\ne0$ and $\wp(u)\ne e_j$.
The addition formulas for elliptic functions give
\[
 \zeta(u-h_j)-\zeta(u)+\zeta(h_j)=\frac{\wp'(u)}{2(\wp(u)-e_j)},
 \qquad
 \zeta(2u)-2\zeta(u)=\frac{\wp''(u)}{2\wp'(u)}.
\]
Therefore,
\eqref{eq:tri3-reduced-critical} becomes
\[
 \frac{\wp'(u)}{\wp(u)-e_j}+\frac{\wp''(u)}{\wp'(u)}=0.
\]
Since $g_2=0$, $g_3=4e_j^3$ and so $(\wp')^2=4\wp^3-4e_j^3$, this is equivalent to
\begin{align*}
 0&=\wp'(u)^2+\wp''(u)(\wp(u)-e_j)
  =4\wp(u)^3-4e_j^3+6\wp(u)^2(\wp(u)-e_j)\\
  &=2(\wp(u)-e_j)
       (5\wp(u)^2+2e_j \wp(u)+2e_j^2).
\end{align*}
Together with $\wp(u)\ne e_j$, we get
\begin{equation}\label{eq:tri3-two-wp-values}
\wp(u)=\frac{-1+3i}{5}e_j
 \quad\text{or}\quad
\wp(u)=\frac{-1-3i}{5}e_j.
\end{equation}
Their absolute values are $\sqrt{2/5}\,e_1<e_1=|e_k|$ for any $k$,
thus neither is a half-period value.
Then the degree-two map $\wp:E_\tau\to\mathbb C\cup\{\infty\}$ has exactly two inverse
images $\{u_{j,\pm},-u_{j,\pm}\}$ over $\frac{-1\pm3i}{5}e_j$ with $u_{j,\pm}\neq -u_{j,\pm}$.
Conversely, every such pair $\{h_j,u_{j,\pm},-u_{j,\pm}\}$ satisfies all the critical point equations \eqref{poho}, hence is a trivial critical point of $G_3$. 
By \eqref{eq:tri3-a-definition} and \eqref{eq:tri3-scaling},
\[
 \wp(\tau^{j-1}a_\pm)
 =\tau^{-2(j-1)}\wp(a_\pm)
 =\frac{-1\pm3i}{5}\tau^{-2(j-1)}e_1
 =\frac{-1\pm3i}{5}e_j,
\]
which implies $\{u_{j,\pm},-u_{j,\pm}\}=\{\tau^{j-1}a_\pm,-\tau^{j-1}a_\pm\}$.
This together with \eqref{eqhjj} proves that \eqref{eq:tri3-seven-configurations} gives all the trivial critical points of $G_3$.

{\bf Step 2. The Hessian at $\mathbf h=\{h_1,h_2,h_3\}$.}

Note that $h_i-h_j=h_k$ in $E_{\tau}$ for $\{i,j,k\}=\{1,2,3\}$. by using \eqref{e0q:tri3-scaling} and \eqref{eqetav}, a direct computation gives
\begingroup
\setlength{\arraycolsep}{4pt}
\renewcommand{\arraystretch}{1.4}
\begin{align*}
&D^2G_3(\mathbf h)=\\
&\frac{1}{2\pi}\begin{pmatrix}
-\beta-4e_1
& 0
& -\beta+\frac{e_1}{2}
& \frac{\sqrt3 e_1}{2}
& -\beta+\frac{e_1}{2}
& -\frac{\sqrt3 e_1}{2}
\\
0
& -\beta+4e_1
& \frac{\sqrt3 e_1}{2}
& -\beta-\frac{e_1}{2}
& -\frac{\sqrt3 e_1}{2}
& -\beta-\frac{e_1}{2}
\\
-\beta+\frac{e_1}{2}
& \frac{\sqrt3 e_1}{2}
& -\beta+2e_1
& -2\sqrt3 e_1
& -\beta-e_1
& 0
\\
\frac{\sqrt3 e_1}{2}
& -\beta-\frac{e_1}{2}
& -2\sqrt3 e_1
& -\beta-2e_1
& 0
& -\beta+e_1
\\
-\beta+\frac{e_1}{2}
& -\frac{\sqrt3 e_1}{2}
& -\beta-e_1
& 0
& -\beta+2e_1
& 2\sqrt3 e_1
\\
-\frac{\sqrt3 e_1}{2}
& -\beta-\frac{e_1}{2}
& 0
& -\beta+e_1
& 2\sqrt3 e_1
& -\beta-2e_1
\end{pmatrix},
\end{align*}
\endgroup
where $\beta=\frac{2\pi}{\sqrt3}.$
Thus, a direct computation gives
\[
\det D^2G_3(\mathbf h)=-\frac{729e_1^6}{16\pi^6}<0.
\]

{\bf Step 3. The Hessian at the other six points.}

By \eqref{eq:tri3-seven-configurations}, \eqref{eqGtau} and $\tau^6=1$, we have
$$\det D^2G_3(\boldsymbol  a_{j,\pm})=\det D^2G_3(\boldsymbol  a_{1,+}),\quad \forall j=1,2,3.$$
Since the addition formulas for $\wp$ give
{\allowdisplaybreaks
\begin{align*}
 x&:=\wp(a_+)=\frac{-1+3i}{5}e_1,\\
 t&:=\wp(h_1-a_+)=\wp(h_1+a_+)
    =e_1+\frac{3e_1^2}{x-e_1}=-(1+i)e_1,\\
 d&:=\wp(2a_+)
    =-2x+\frac{\wp''(a_+)^2}{4\wp'(a_+)^2}
    =-2x+\frac{9x^4}{4x^3-4e_1^3}
    =\frac{1-8i}{5}e_1,
\end{align*}
}%
a direct computation leads to
\begin{equation}\label{eq:tri3-pair-hessian}
D^2G_3(h_1,a_+,-a_+)=\frac{1}{2\pi}\begin{pmatrix}P&Q&Q\\Q&R&T\\Q&T&R\end{pmatrix},
\end{equation}
where
{\allowdisplaybreaks
\begin{align*}
 P&
   =-\beta I_2+e_1\begin{pmatrix}-5&2\\2&5\end{pmatrix},\quad
 Q
   =-\beta I_2+e_1\begin{pmatrix}1&-1\\-1&-1\end{pmatrix},\\
 R&
   =-\beta I_2+\frac{e_1}{5}\begin{pmatrix}-1&22\\22&1\end{pmatrix},\quad
 T
   =-\beta I_2+\frac{e_1}{5}\begin{pmatrix}-1&-8\\-8&1\end{pmatrix},
\end{align*}
}%
and $\beta=\frac{2\pi}{\sqrt3}$, and $I_2$ is the $2\times 2$ identity matrix. Thus, a direct computation gives
\[
 \det D^2G_3(\boldsymbol  a_{j,\pm})=\det D^2G_3(\boldsymbol  a_{1,+})
 =-\frac{729e_1^4}{10\pi^6}
       \left(e_1^2-\frac{10\pi^2}{3}\right)<0, 
\]
where we have used \eqref{eq:tri3-e-bound} to obtain the last inequality.
  The proof is complete.
\end{proof}

\subsection*{Acknowledgements} 
Z. Chen was supported by National Key R\&D Program of China (Grant 2023YFA1010002) and NSFC (No. 12222109). The AI tool ChatGPT was used in the preparation of this manuscript. The authors assume full responsibility for the
mathematical validity, accuracy, and integrity of the proofs presented in the manuscript.

\end{document}